\documentclass[11pt]{article}
\usepackage[T1]{fontenc}
\usepackage[utf8]{inputenc}
\usepackage{lmodern,microtype}
\usepackage[margin=1in]{geometry}
\usepackage{amsmath,amssymb,amsthm,mathtools,booktabs,array}
\usepackage{xurl,needspace}
\usepackage{tikz}
\usetikzlibrary{arrows.meta,positioning,calc,decorations.pathreplacing,patterns}
\usepackage[hidelinks,pagebackref]{hyperref}
\newcommand{\doi}[1]{\href{https://doi.org/#1}{doi:\nolinkurl{#1}}}
\renewcommand*{\backref}[1]{}
\renewcommand*{\backrefalt}[4]{\ifcase #1\or\space(Cited on p.~#2.)\else\space(Cited on pp.~#2.)\fi}
\hypersetup{pdftitle={Variance Bounds and Strict Log-Concavity for Forest Coronas},pdfauthor={Vadim E. Levit}}

\newtheorem{theorem}{Theorem}[section]
\newtheorem{lemma}[theorem]{Lemma}
\newtheorem{proposition}[theorem]{Proposition}
\newtheorem{corollary}[theorem]{Corollary}
\theoremstyle{definition}

\theoremstyle{remark}\newtheorem{remark}[theorem]{Remark}
\newcommand{\E}{\mathbb E}
\newcommand{\Var}{\operatorname{Var}}

\newcommand{\Prob}{\mathbb P}

\title{Variance Bounds and Strict Log-Concavity for Forest Coronas}
\author{Vadim E. Levit\\[2pt]
\small Department of Mathematics\\
\small Ariel University\\
\small Israel\\[2pt]
\small\texttt{levitv@ariel.ac.il}}
\date{October 27, 2026}
\begin{document}
\maketitle
\begin{abstract}
We prove that the size of a random independent set in every nonempty
forest has variance strictly smaller than its mean at common activity
$0<t\le1$. A rooted quadratic reserve gives quantitative forms of this
inequality. Mean matching transfers the bounds to fixed-weight canonical
completion laws, and exact binomial moments prove strict log-concavity
for every forest corona and every well-covered forest. The stronger
factorial normalization $((m-j)!i_j)_{j=0}^m$ is strictly log-concave;
we give explicit coefficient margins, terminal sharpness, one-less
smoothing, and the exact universal three-index binomial ULC prefix.
An exact attachment argument doubles the rooted reserve for every
nontrivial unweighted tree. For a single marked root, a square reserve holds through
$\Theta=\tfrac12e^{3-2W(1/e)}$ and for every positive mark at roots of
degree at most four. A stronger reserve holds through mark $15/14$.
These estimates yield marked clique, adjacent-pair, and repulsive
two-piece variance bounds with fixed-weight canonical transfer.
Finally, arbitrary positive site activities in $[3t/4,t]$ satisfy a
uniform rooted reserve on every forest. This proves strict log-concavity
for weighted forest coronas with compatible base marks. An explicit
weighted-star obstruction shows why an unrestricted heterogeneous
extension is impossible. Complete proofs are included, together with
the required classical graph characterization and variance comparison.
\end{abstract}

\section{Introduction}\label{sec:introduction}
For a finite simple graph $G$, write
\[
 I(G;x)=\sum_{j=0}^{\alpha(G)}i_j(G)x^j,
\]
where $i_j(G)$ counts independent sets of size $j$. The independence
polynomial was introduced by Gutman and Harary~\cite{GutmanHarary1983}.
A positive coefficient sequence is strictly log-concave if
$i_j^2>i_{j-1}i_{j+1}$ at every internal index. The hard-core law at
activity $t>0$ assigns an independent set $S$ probability
$t^{|S|}/I(G;t)$. We study the size $K=|S|$ through its mean and variance.

Our starting result is a forest-specific upper variance bound. If a
tree is rooted at $o$, let $q$ be the root occupation probability and
let $d$ be the difference between the conditional means of $K$ given
that $o$ is occupied and unoccupied. Lemma~\ref{lem:forest-subpoisson}
proves
\[
 \E_tK-\Var_tK\ge\frac{t}{1+t}\,q d^2\qquad(0<t\le1),
\]
with strict inequality for a tree having more than one vertex. The
response $d$ may have either sign. The proof retains a second reserve
after deleting the root; this is what allows induction through arbitrary
branching. Summing over components gives $\Var_tK<\E_tK$ for every
nonempty forest. Corollaries~\ref{cor:explicit-variance-reserve} and
\ref{cor:mean-sensitive-variance-reserve} quantify the remaining deficit.

The coefficient consequence is uniform in graph order. A graph is
well covered when all its maximal independent sets have the same size.
The leaf corona $H\circ K_1$ attaches one private leaf to each vertex
of $H$. Every well-covered forest is a disjoint union of such a corona
and isolated vertices; Section~\ref{sec:representation} includes the
classical characterization and its small cases.
\begin{theorem}[Every well-covered forest]\label{thm:allorder-wc}
Let $F$ be a well-covered forest and put $m=\alpha(F)$. At every
internal index,
\begin{equation}\label{eq:main-forest-ratio}
 \frac{i_j(F)^2}{i_{j-1}(F)i_{j+1}(F)}>
 \frac{m-j+1}{m-j}>1\qquad(1\le j<m).
\end{equation}
Consequently, the sequence $((m-j)!i_j(F))_{j=0}^m$ is strictly
log-concave. The conclusions hold after every common positive
coefficient tilt $i_j(F)\mapsto t^j i_j(F)$.
\end{theorem}

The conclusion includes disconnected forests and arbitrary isolated
vertices. Reversing the sequence gives strict ultra-log-concavity of
order infinity: with $A_q=i_{m-q}(F)$, the inequality is
$qA_q^2>(q+1)A_{q-1}A_{q+1}$. This differs from full binomial
ultra-log-concavity, whose curvature factor also contains $(j+1)/j$.
Theorem~\ref{thm:exact-ulc-prefix} determines the universal fixed indices
for that stronger normalization: precisely the internal indices
$1,2,3$. Every fixed index $j\ge4$ fails for infinitely many
well-covered trees.

The variance method also gives quantitative margins. For
$D_{m,j}=(m-(j+1)/2)^2+j+1$, Corollary~\ref{cor:improved-universal-margin}
gives
\[
 \frac{i_j(F)^2}{i_{j-1}(F)i_{j+1}(F)}>
 \frac{m-j+1}{m-j}\frac{D_{m,j}}{D_{m,j}-1}.
\]
A component-sensitive alternative appears in
Corollary~\ref{cor:component-coefficient-gap}. The leading terminal
factor is sharp at every fixed offset, by
Proposition~\ref{prop:terminal-factor-sharpness}. We do not assert
optimality of the finite-size corrections.

For the clique corona $H\circ K_p$, attach a private $p$-clique at
each base vertex and join that vertex to every vertex of the clique.
The same transfer gives the following result.
\begin{corollary}[Every forest clique corona]\label{cor:forest-clique-coronas}
Let $H$ be a nonempty forest of order $h$, and let $p\ge1$ be an
integer. For every $1\le j<h$,
\begin{equation}\label{eq:clique-ratio}
 \frac{i_j(H\circ K_p)^2}
      {i_{j-1}(H\circ K_p)i_{j+1}(H\circ K_p)}>
 \frac{h-j+1}{h-j}\frac{p^2(j+1)}{p^2j+1}.
\end{equation}
In particular, its independence polynomial is strictly log-concave.
\end{corollary}

Theorem~\ref{thm:one-less-smoothing} further shows that
$(1+x)^{h-1}I(H;x/(1+x))$ is strictly log-concave whenever the forest
$H$ has an edge. This is a polynomial quotient statement; it does not
identify deletion of an arbitrary private leaf.

\paragraph{Fixed marks and boundary laws.}
A mark is a positive multiplier held fixed while the common activity
varies. Theorem~\ref{thm:marked-square-reserve} proves
$\E K-\Var K\ge q^2d^2$ when only one root is marked by
$0<\theta\le\Theta=\tfrac12\exp(3-2W(1/e))$, where
$W(z)e^{W(z)}=z$ and $W(z)\ge0$. At root degree at most four,
Corollary~\ref{cor:low-degree-mark} allows every positive mark.
Theorem~\ref{thm:strong-marked-reserve} retains the stronger ordinary
reserve through mark $15/14$. The marked results and the logarithmic
deletion floors supply reusable inputs for conditional assembly.
We prove clique assembly, the adjacent OR-event law, and repulsive
soft-edge coupling of two marked pieces. In particular, suppressing
the event that both designated roots are absent gives a strictly
sub-Poisson law on the stated fixed-mark interval.

Lemma~\ref{lem:capacity-transfer} transfers any such fixed-weight
variance bound to completion weights
$w(S)\binom{N-|S|}{r-|S|}$ for every $N\ge\alpha(G)$ and
$1\le r\le N$. All original weights remain fixed during mean
matching. The capacity assumption concerns the independence number;
an order bound $N\ge|V(G)|$ is not required for this probability lemma.

\paragraph{Compatible heterogeneous activities.}
Theorem~\ref{nuf:thm:main} extends the rooted induction to arbitrary
vertex activities in $[3t/4,t]$, $0<t\le1$. No restriction on
branching, the number of marked vertices, or graph order is imposed.
Fixed-mark canonical transfer follows, and
Corollary~\ref{nuf:cor:corona} proves full strict log-concavity for
weighted forest coronas when the base marks are at most one and their
largest-to-smallest ratio is at most $4/3$. The quantitative theorem
also gives a larger two-parameter sufficient region. The weighted-star
examples in Section~\ref{sec:heterogeneous-limit} show that bounding
every site activity by one alone is insufficient.

\paragraph{Relation to earlier work.}
Levit and Mandrescu~\cite[Theorems 2.4, 2.5, and 2.7]{LM2004}
proved log-concavity for several corona classes, including star and
path coronas; later work studies their roots and transforms
\cite{LMRoots2008,LMCorona2016}. Brown and Cameron~\cite[Theorem 2.4]{BC2018}
prove strict log-concavity after sufficiently many iterated leafy
extensions of each fixed graph. Here one extension suffices for
every forest base. Zhu~\cite{BZhuClique2016} proves clique-corona
log-concavity when the base polynomial is log-concave. The present
forest theorem does not assume that hypothesis: tree independence
polynomials need not be log-concave~\cite{KL26,GalvinFailures}.

The published general clique-corona bound of Hoang, Levit, Mandrescu,
and Pham~\cite[Corollary 3.10]{HLMP2026} applies when
$p\ge n^2/[4(n+1)]$, where $n$ is the base order. Our forest result
applies to every $p\ge1$. The unrestricted assertions in an earlier
version~\cite[Corollaries 3.10 and 3.15]{HLMP2024v1} are not used here.
The variance comparison is the finite-support case of Yu's
relative-log-concavity order~\cite[Theorem 12]{Yu2010}; its required
two-crossing proof is included. For a real-rooted independence
polynomial, the hard-core size is a sum of Bernoulli variables and
strict sub-Poisson variance is immediate. Forest bases need not be
real-rooted, and the present rooted reserves apply without that
assumption.

Hoang, Pham, and Trung~\cite[Section 7, pp.~2550343-16--2550343-20]{HPTCorona2026}
study $h$-polynomials of corona independence complexes. Their
Corollary 7.9 gives real-rootedness for claw-free bases, and their
Theorem 7.11 proves unimodality when the base $h$-polynomial has
nonnegative coefficients and the attachment size is at least its degree.
Those hypotheses do not hold for arbitrary forest bases. Their known
$h$-transform gives the application in
Subsection~\ref{subsec:corona-h-polynomials}; the rooted forest variance
reserve and quantitative independence-coefficient theorem proved here
are separate from that transform.

There is also a general-graph overlap at very small activity.
For a graph of order $n$, the proof of Davies, Sandhu, and Tan~\cite[Theorem 1.5]{DaviesSandhuTan2025}
gives $V\le a-na^2+(n-1)sa$, where $a=\E_tK/n$,
$V=\Var_tK/n$, and $s=t/(1+t)$. Using their bound
$a\ge t/(1+nt)$, we deduce, for $n\ge3$ and
$0<t<1/[n(n-2)]$,
\[
 \E_tK-\Var_tK\ge
 \frac{nt^2(1-n(n-2)t)}{(1+nt)^2(1+t)}>0.
\]
This order-dependent interval and unrooted reserve do not supply the
forest-wide rooted inequality on $(0,1]$ needed here.

\paragraph{Organization and scope.}
Sections~\ref{sec:representation}--\ref{sec:transfer} prove the
unweighted forest and coefficient results. Section~\ref{sec:marked-boundaries}
contains the single-root and two-piece fixed-mark bounds, and
Section~\ref{nuf:sec:forest-ratio} treats compatible heterogeneous
activities. Section~\ref{cld:sec:doubled-reserve} proves the doubled
reserve by three attachment lemmas and strong induction. The proofs
use no graph enumeration and no minimum order.
Path interaction certificates, graph-specific edge restorations, and
their coefficient comparisons are separate downstream applications.
Neither sub-Poisson variance at common activity nor the canonical
completion theorem alone asserts log-concavity for every forest.

\paragraph{Relation to the barrier study.}
The first public version of~\cite{LevitBL2026} proves eventual barrier
log-concavity and corresponding eventual corona consequences. The present
paper gives a self-contained rooted-variance and canonical-transfer route
to all-order well-covered-forest and forest-corona results. The expanded
barrier manuscript cites this paper for these developments; the shared
classical corona identities retain their original attribution. General
forest barrier log-concavity remains open.

\section{Coronas and well-covered forests}\label{sec:representation}
All graph deletions are induced. Write $\mu(G)$ for the maximum
matching size and $z(G)$ for the number of isolated vertices.
The leaf-corona identity and the well-covered-tree characterization
are classical~\cite[Theorem 2.3 and the paragraph preceding Theorem 2.6]{LM2004};
see also~\cite{FHN,BC2018}. Levit and Mandrescu~\cite{LM1999}
also characterized well-covered trees through well-covered spiders
and internal edge-joins. We give the representation needed here
in full.
\begin{figure}[tbp]
\centering
\begin{tikzpicture}[x=.9cm,y=.9cm,>=Latex,font=\small,
 old/.style={circle,fill=black,inner sep=1.9pt},
 leaf/.style={circle,draw,fill=white,inner sep=1.9pt}]
\node at (2,1.1) {$H=P_3\sqcup K_1$};
\foreach \i/\x in {a/0,b/1.4,c/2.8,d/4.2}{\node[old] (l\i) at (\x,0) {};}
\draw (la)--(lb)--(lc);
\draw[->,thick] (5,.05)--(6.5,.05) node[midway,above,align=center] {one leaf\\per vertex};
\node at (10.4,1.1) {$F=(H\circ K_1)\sqcup K_1$};
\foreach \i/\x in {a/7.3,b/8.7,c/10.1,d/11.5}{
 \node[old] (r\i) at (\x,0) {};
 \node[leaf] (p\i) at (\x,-.95) {};
 \draw (r\i)--(p\i);
}
\draw (ra)--(rb)--(rc);
\node[leaf] (iso) at (13.2,0) {};
\node[below,align=center] at (13.2,-.25) {isolated\\vertex};
\node[align=center] at (2,-1.5) {$|V(H)|=4$};
\node[align=center] at (10.2,-1.8) {$\mu(F)=4,\quad z(F)=1,\quad |V(F)|=9$};
\end{tikzpicture}
\caption{The well-covered-forest decomposition. Filled vertices are the
base vertices; open vertices are attached leaves or isolated components.
Here $z(F)$ counts isolated vertices and $\mu(F)$ is the maximum matching size.
In general, $F=(H\circ K_1)\sqcup sK_1$ has
$I(F;x)=(1+x)^{|V(H)|+s}I(H;x/(1+x))$ and
$|V(F)|-z(F)=2\mu(F)=2|V(H)|$. The all-order theorem applies to every such representation.}
\label{fig:corona-decomposition}
\end{figure}
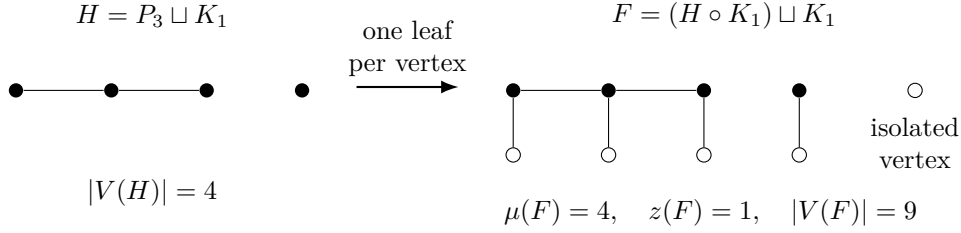

\begin{lemma}[Classical well-covered-tree representation]\label{lem:wc-representation}
Every well-covered tree $T\ne K_1$ has the form $H\circ K_1$ for a tree
$H$, with $|V(T)|=2|V(H)|$. Conversely, every such corona is well-covered.
\end{lemma}
\begin{proof}
The characterization is classical \cite{FHN,LM2004}; we give its tree
proof including the small cases. In a nontrivial tree, both bipartition
classes are maximal independent sets. Well-coveredness makes both have
size $\alpha(T)$. If they are $L,R$, then
$S\cup(R\setminus N(S))$ is independent for $S\subseteq L$, giving
$|N(S)|\ge|S|$. Hall's theorem supplies a perfect matching.
Every matching edge must be pendant: if $uv$ is not, choose neighbors
$u'\ne v$ of $u$ and $v'\ne u$ of $v$. Acyclicity makes $u',v'$
distinct and nonadjacent. A maximal independent set extending
$\{u',v'\}$ misses both $u,v$ and has at most one vertex in every other
matching edge. Its size is at most $\alpha(T)-1$, a contradiction.
For $T=K_2$, use $H=K_1$. Otherwise, every nonleaf has exactly one leaf
neighbor, and the induced subgraph of nonleaves is connected and acyclic.
It is $H$. Conversely, every maximal independent set in $H\circ K_1$
contains exactly one vertex from each old-vertex/private-leaf pair.
\end{proof}

\begin{lemma}[Forest representation and corona identity]
\label{lem:forest-corona-identity}
Every well-covered forest has the form
$F=(H\circ K_1)\sqcup sK_1$, where $H$ is a forest. If $h=|V(H)|$,
then $\alpha(F)=h+s$ and
\[
 I(F;x)=(1+x)^{h+s}I\!\left(H;\frac{x}{1+x}\right).
\]
\end{lemma}
\begin{proof}
A maximal independent set in a disjoint union restricts to a maximal
independent set in each component, and arbitrary componentwise choices
combine to a maximal independent set. Hence a forest is well-covered
if and only if each component is well-covered. Apply
Lemma~\ref{lem:wc-representation} to each nontrivial component and
collect the isolated vertices. Every maximal independent set contains
one vertex from each pendant pair and every isolated vertex, so its
size is $h+s$. For any graph $H$ of order $h$, conditioning on its
selected independent set gives the classical identity
\cite[Theorem 2.3]{LM2004}
\begin{equation}\label{eq:corona-identity}
 I(H\circ K_1;x)=(1+x)^h I\!\left(H;\frac{x}{1+x}\right).
\end{equation}
The isolated vertices supply the factor $(1+x)^s$.
\end{proof}

\section{An analytic variance inequality for forests}\label{sec:variance}
The hard-core law at common activity $t>0$ gives each independent
set $J$ probability $t^{|J|}/I(H;t)$. We write $K=|J|$ and use
$\E_t$ and $\Var_t$ for its moments. A rooted quadratic reserve
controls signed conditional responses without imposing a sign on
the response of an individual vertex. The proof has three steps:
Lemma~\ref{lem:rooted-scalar-positivity} supplies a purely scalar
positivity estimate; the rooted induction in
Lemma~\ref{lem:forest-subpoisson} converts it into a variance reserve;
and Theorem~\ref{thm:general-transfer} converts the variance bound into
coefficient inequalities by mean matching. Lemma~\ref{lem:forest-corona-identity}
identifies the transform for well-covered forests, giving
Theorem~\ref{thm:allorder-wc}. We state the rooted conclusion first and give
the complete scalar argument immediately after its proof.

\begin{lemma}[A sub-Poisson bound for forests]
\label{lem:forest-subpoisson}
Let $H$ be a nonempty finite forest, and let $K$ be the size of its
hard-core independent set at common activity $0<t\le1$. Then
\[
 \Var_t K<\E_t K.
\]
More precisely, root a tree $T$, let $q$ be its root occupation
probability, and put
\[
 a=\frac{t}{1+t},\qquad
 d=\E_t[K\mid\text{root occupied}]
   -\E_t[K\mid\text{root unoccupied}].
\]
Then
\begin{equation}\label{eq:rooted-subpoisson-invariant}
 \E_t K-\Var_t K\ge aqd^2.
\end{equation}
Equality in \eqref{eq:rooted-subpoisson-invariant} holds for a
one-vertex tree; the inequality is strict for every larger tree.
\end{lemma}
\begin{proof}
At a fixed common activity $0<t\le1$, the induction is asserted at
\emph{every} choice of root. Root deletion is a disjoint union of
smaller rooted trees; the logarithmic deletion reserve used below
is derived from those smaller cases. Thus the two reserves are
proved together by induction on the number of vertices.
Write $g_T=\E_tK-\Var_tK$. We prove the rooted assertion by induction
on the order of $T$. For a single vertex, $q=a$, $d=1$, and $g_T=a^2$.
For the induction step, let $T_i$ be the child trees, with root
quantities $q_i,d_i$, and write $g_i=g_{T_i}$ and
$g_i^-=g_{T_i-\mathrm{root}}$. Means and variances add over the
components of the latter forest. Set
\[
 s=1-q,\qquad S=\sum_iq_i,\qquad u=\sum_iq_id_i.
\]
The root odds and the conditional-mixture formula give
\begin{equation}\label{eq:subpoisson-mixture}
 \frac q{1-q}=t\prod_i(1-q_i),\qquad d=1-u,
 \qquad
 g_T=s\sum_i g_i+q\sum_i g_i^-+q-qs d^2.
\end{equation}
The added $q$ accounts for the occupied root in the second conditional
law.

Put
\[
 L_i=\log\frac{t(1-q_i)}{q_i}.
\]
If $q_{ij},d_{ij}$ are the root quantities in the child trees of $T_i$,
then
\[
 L_i=\sum_j-\log(1-q_{ij}),\qquad
 1-d_i=\sum_jq_{ij}d_{ij},\qquad \sum_jq_{ij}\le L_i.
\]
The induction hypothesis and Cauchy--Schwarz therefore give
\begin{equation}\label{eq:subpoisson-child-bounds}
 g_i\ge aq_id_i^2,
 \qquad
 g_i^-\ge a\sum_jq_{ij}d_{ij}^2
       \ge a\frac{(1-d_i)^2}{L_i}.
\end{equation}
When $L_i=0$, the child is a leaf and $d_i=1$ exactly. In every formula
involving $1/L_i$, this case imposes the constraint $d_i=1$; no division
by zero is required. Subtracting $aqd^2$ in
\eqref{eq:subpoisson-mixture}, it is enough to prove that
\begin{equation}\label{eq:subpoisson-quadratic}
 as\sum_iq_id_i^2+aq\sum_i\frac{(1-d_i)^2}{L_i}
       +q-q(s+a)(1-u)^2>0.
\end{equation}
Here the root has at least one child, so $S>0$.

Each child root satisfies $0<q_i\le a$ by its odds relation.
Lemma~\ref{lem:rooted-scalar-positivity} below therefore applies to
\eqref{eq:subpoisson-quadratic}, with the required constraint $d_i=1$
whenever $L_i=0$. This completes the induction.
Every nonempty tree has a strictly positive deficit: for a one-vertex
tree it is $a^2$, and for a larger tree it is strictly greater than
$aqd^2\ge0$.
Adding these deficits over the components proves the forest assertion.
\end{proof}

\begin{lemma}[Scalar positivity for rooted messages]
\label{lem:rooted-scalar-positivity}
Let $0<t\le1$, put $a=t/(1+t)$, and let $q_1,\ldots,q_k$ satisfy
$k\ge1$ and $0<q_i\le a$. Define $q\in(0,1)$, $s$, and $L_i$ by
\[
 \frac{q}{1-q}=t\prod_{i=1}^k(1-q_i),\qquad s=1-q,
 \qquad L_i=\log\frac{t(1-q_i)}{q_i}\ge0.
\]
Let $d_1,\ldots,d_k$ be arbitrary real numbers subject to $d_i=1$
when $L_i=0$, and put $u=\sum_iq_id_i$. Then
\[
 as\sum_iq_id_i^2+aq\sum_{i:L_i>0}\frac{(1-d_i)^2}{L_i}
       +q-q(s+a)(1-u)^2>0.
\]
In the proof, a term $(1-d_i)^2/L_i$ with $L_i=0$ is interpreted as
zero under the stated constraint; no division by zero is performed.
\end{lemma}
\begin{proof}
Put $S=\sum_iq_i>0$. Define
\[
 m=\sum_i\frac{q_iq}{q+s q_iL_i},\qquad
 V=\frac{S-m}{as},\qquad H=as m,
 \qquad C=s+a,\qquad b=qC.
\]
For $L_i>0$, put $A_i=a(sq_i+q/L_i)$ and $B_i=aq/L_i$.
Completing squares and applying weighted Cauchy--Schwarz show that
\begin{equation}\label{eq:subpoisson-square-bound}
 as\sum_iq_id_i^2+aq\sum_i\frac{(1-d_i)^2}{L_i}
 \ge H+\frac{(u-m)^2}{V}
 \quad\text{if }V>0.
\end{equation}
Indeed,
\[
 \sum_i\frac{q_iB_i}{A_i}=m,\qquad
 \sum_i\frac{q_i^2}{A_i}=V,\qquad
 \sum_i\left(B_i-\frac{B_i^2}{A_i}\right)=H.
\]
An index with $L_i=0$ contributes $q_i$ to $m$, zero to $V$, and $asq_i$ to $H$,
which are the finite limits of these expressions. If $V=0$, every
index has $L_i=0$, so $u=m=S$ is fixed, and the left-hand side of
\eqref{eq:subpoisson-square-bound} equals $H$.

It suffices to prove
\begin{equation}\label{eq:subpoisson-determinant}
 F:=(q+H)(1-bV)-b(1-m)^2>0.
\end{equation}
In fact, \eqref{eq:subpoisson-determinant} implies $1-bV>0$ because
$q+H>0$. When $V>0$, completing the remaining square gives
\[
 H+q+\frac{(u-m)^2}{V}-b(1-u)^2
 =\frac{F}{1-bV}
 +\frac{1-bV}{V}
   \left(u-\frac{m-bV}{1-bV}\right)^2.
\]
Thus its minimum over real $u$ is $F/(1-bV)>0$.
When $V=0$, its constrained value is $F>0$.
Direct expansion of \eqref{eq:subpoisson-determinant} gives
\begin{equation}\label{eq:subpoisson-expanded-F}
 F=q(q-a)-\frac{q^2CS}{as}
   +m\left[as+qC\left(2+\frac q{as}-S\right)\right].
\end{equation}

For the scalar estimates, put
\[
 L=\log\frac{t(1-q)}q=\sum_i-\log(1-q_i),\qquad
 c=\frac a{\log(1+t)},\qquad w=\frac te,\qquad
 A=as,\qquad R=\frac{qC}{A}.
\]
Here $e=\exp(1)$. Since $0<q_i\le a$ and
$-\log(1-x)/x$ is increasing,
\begin{equation}\label{eq:subpoisson-message-bounds}
 cL\le S\le L.
\end{equation}
Moreover,
\[
 q_iL_i
 =q_i\log\frac{t(1-q_i)}{q_i}
 \le q_i\log\frac t{q_i}\le\frac te=w,
\]
where the last bound follows by maximizing $x\log(t/x)$ over $x>0$.
Consequently,
\begin{equation}\label{eq:subpoisson-m-bound}
 m\ge\frac{Sq}{q+sw}.
\end{equation}
The defining odds relation also gives
\begin{equation}\label{eq:subpoisson-odds-bounds}
 R=(1+t)Ce^{-L}\le(1+2t)e^{-L},\qquad
 a-q=A(1-e^{-L}).
\end{equation}

The coefficient $P$ of $m$ in
\eqref{eq:subpoisson-expanded-F} is positive. Indeed,
\[
 \frac PA=1+R\left(2+\frac qA-S\right)\ge1+R(2-S).
\]
For $S\le2$, this is positive; for $S>2$, using $L\ge S$ and
$(S-2)e^{-S}\le e^{-3}$ gives
\[
 \frac PA\ge1-(1+2t)e^{-3}\ge1-3e^{-3}>0.
\]
We may therefore substitute \eqref{eq:subpoisson-m-bound} into
\eqref{eq:subpoisson-expanded-F}. Exact simplification yields
\begin{equation}\label{eq:subpoisson-G}
 F\ge\frac q{q+sw}G,\qquad
 G=(q-a)(q+sw)+S\left[A+qC\left(2-S-\frac wa\right)\right].
\end{equation}
Set $\beta=w/a=(1+t)/e$ and $\delta=2-\beta$. By
\eqref{eq:subpoisson-odds-bounds},
\begin{equation}\label{eq:subpoisson-G-normalized}
 \frac GA=S-(1-e^{-L})(q+sw)+RS(\delta-S).
\end{equation}
We use only the elementary estimates
\begin{equation}\label{eq:subpoisson-numerical-constants}
 e>\frac83,\qquad \log2<\frac7{10}.
\end{equation}
The first follows from the positive exponential series after its cubic
term. For the second,
$e^{7/10}>1+7/10+(7/10)^2/2+(7/10)^3/6>2$.
In particular, $\beta<3/4$ and $\delta>5/4$.

If $S\le\delta$, the final term in
\eqref{eq:subpoisson-G-normalized} is nonnegative. Since $w<1$ and
$q\le a$,
\[
 q+sw\le a+(1-a)w.
\]
Using $1-e^{-L}\le L\le S/c$ and
\eqref{eq:subpoisson-numerical-constants}, we obtain
\[
 \frac GA
 \ge S\left(1-\frac{a+(1-a)w}{c}\right)
 =S\left(1-\left(1+\frac1e\right)\log(1+t)\right)
 >S\left(1-\frac{77}{80}\right)>0.
\]
If $S>\delta$, then
\[
 \frac{(1-e^{-L})(q+sw)}S
 <\frac{a+(1-a)w}{\delta}
 <\frac{11/16}{5/4}=\frac{11}{20},
\]
because
$a+(1-a)w=t(1+1/e)/(1+t)\le(1+1/e)/2<11/16$.
Also, \eqref{eq:subpoisson-odds-bounds} and $L\ge S$ imply
\[
 R(S-\delta)
 \le3e^{-S}(S-\delta)
 \le3e^{-\delta-1}
 =3e^{\beta-3}
 <\frac3{e^2}<\frac{27}{64}.
\]
Here the maximum of $(S-\delta)e^{-S}$ for $S\ge\delta$ is
$e^{-\delta-1}$. Thus
\[
 \frac GA>S\left(1-\frac{11}{20}-\frac{27}{64}\right)
 =\frac9{320}S>0.
\]
In both cases, $G>0$, so \eqref{eq:subpoisson-G} proves
\eqref{eq:subpoisson-determinant}, and hence the asserted strict positivity.
\end{proof}

\begin{corollary}[An explicit variance reserve]\label{cor:explicit-variance-reserve}
Let $H$ be a nonempty forest of order $h$, and let $\ell$ be the
largest order of one of its components. At every common activity
$0<t\le1$, put $\mu=\E_tK$ and $V=\Var_tK$. Then
\begin{equation}\label{eq:explicit-variance-reserve}
 \mu-V>\frac{V^2}{h\ell}.
\end{equation}
\end{corollary}
\begin{proof}
Put $a=t/(1+t)$ and $g=\mu-V$. For each component $C$, let
$g_C=\E_t|J\cap V(C)|-\Var_t|J\cap V(C)|$. For each vertex $v$,
write $q_v=\Prob_t(v\in J)$ and let $d_v$ be the difference between
the conditional means of $K$ given that $v$ is occupied and
unoccupied. Rooting $C$ at $v$ in Lemma~\ref{lem:forest-subpoisson}
gives $g_C\ge aq_vd_v^2$. Independence between components therefore
gives
\[
 \sum_vq_vd_v^2\le\frac1a\sum_C|V(C)|g_C
 \le\frac{\ell g}{a}.
\]
Since $K=\sum_v\mathbf1_{\{v\in J\}}$, the covariance identity is
$V=\sum_vq_v(1-q_v)d_v$. Weighted Cauchy--Schwarz yields
\[
 V^2\le
 \left(\sum_vq_v(1-q_v)\right)
 \left(\sum_vq_v(1-q_v)d_v^2\right)
 <\mu\sum_vq_vd_v^2
 \le\frac{\mu\ell g}{a}\le h\ell g.
\]
Here $V>0$ and all $q_v>0$, which justify the strict inequality.
The final bound uses $q_v\le a$, hence $\mu\le ha$. This occupation
bound follows because, conditional on every neighbor of $v$ being
unoccupied, its occupation probability is $a$, and otherwise it is zero.
\end{proof}

\begin{corollary}[A mean-sensitive variance reserve]
\label{cor:mean-sensitive-variance-reserve}
Let $H$ be a nonempty forest of order $h$, and let $K$ be its
hard-core independent-set size at common activity $0<t\le1$.
For each component $C$, put
$n_C=|V(C)|$ and $\mu_C=\E_t|J\cap V(C)|$, and write
$\mu=\E_tK$ and $V=\Var_tK$. Then
\begin{equation}\label{eq:mean-sensitive-variance-reserve}
 \mu-V\ge
 \frac{V^2}{\sum_C(n_C-\mu_C)^2}
 \ge\frac{V^2}{(h-\mu)^2}.
\end{equation}
The first inequality is an equality exactly when $H$ has no edges,
and the overall inequality between $\mu-V$ and $V^2/(h-\mu)^2$ is strict if $h>1$.
If $\ell=\max_C n_C$, then also
\[
 \mu-V\ge\frac{V^2}{\ell(h-\mu)}.
\]
\end{corollary}
\begin{proof}
Put $a=t/(1+t)\le1/2$. For a component $C$, abbreviate
$n=n_C$, $\mu_C=\mu_0$, $V_C=\Var_t|J\cap V(C)|$, and
$g_C=\mu_0-V_C$. For each $v\in V(C)$, let $q_v$ be its occupation
probability and $d_v$ the difference of the conditional means of
$|J\cap V(C)|$ given that $v$ is occupied and unoccupied.
Lemma~\ref{lem:forest-subpoisson}, applied with each possible root,
gives
\[
 \sum_{v\in V(C)}q_vd_v^2\le\frac{ng_C}{a},
\]
strictly if $n>1$. The covariance identity and Cauchy--Schwarz give
\[
 V_C^2
 =\left(\sum_{v\in V(C)}q_v(1-q_v)d_v\right)^2
 \le\left(\sum_{v\in V(C)}q_v(1-q_v)^2\right)
      \left(\sum_{v\in V(C)}q_vd_v^2\right).
\]
We have $0<q_v\le a\le1/2$, and $x(1-x)^2$ is concave on
$[0,1/2]$, since its second derivative is $-4+6x<0$ there.
Jensen's inequality and $\mu_0\le na$ consequently yield
\[
 V_C^2
 \le \mu_0\left(1-\frac{\mu_0}{n}\right)^2\frac{ng_C}{a}
 \le(n-\mu_0)^2g_C.
\]
This inequality is strict if $n>1$; equality holds for $n=1$.
The components are independent, so another Cauchy--Schwarz inequality
gives
\[
 V=\sum_C V_C
 \le\sum_C(n_C-\mu_C)\sqrt{g_C}
 \le\sqrt{\left(\sum_C(n_C-\mu_C)^2\right)(\mu-V)}.
\]
This proves the first inequality in
\eqref{eq:mean-sensitive-variance-reserve}, including strictness when an edge is present. If $H$ has no edges,
all components are identical single-vertex laws, so both
Cauchy--Schwarz bounds are equalities. Since all $n_C-\mu_C$ are positive,
\[
 \sum_C(n_C-\mu_C)^2\le(h-\mu)^2,
\]
strictly when $H$ has more than one component. Thus the final bound
is strict for $h>1$. Finally,
\[
 \sum_C(n_C-\mu_C)^2
 \le\sum_C n_C(n_C-\mu_C)\le\ell(h-\mu),
\]
which proves the last assertion.
\end{proof}

\section{Quantitative transfer from variance to coefficients}\label{sec:transfer}
The following comparison is the finite-support variance case of the
classical relative-log-concavity order; see Yu
\cite[Theorem 12]{Yu2010} and the earlier two-crossing results cited there.
We include its short proof for completeness.

\begin{lemma}[A two-crossing variance comparison]
\label{lem:relative-lc-variance}
Let $p$ be a probability mass function that is positive on
$\{0,\ldots,d\}$, and let $\nu$ be another
probability law on this set. Suppose $\nu/p$ is log-concave with interval
support, and the two laws have the same mean. Then
$\Var_\nu K\le\Var_p K$.
\end{lemma}
\begin{proof}
The indices where $\nu(k)-p(k)>0$ form an interval. If the laws differ,
equal total mass and equal mean force negative differences on both sides
of this interval. Choose cutpoints $a\le b$ between the sign changes.
Then $(\nu(k)-p(k))(k-a)(k-b)\le0$ for every $k$. On summing, the
constant and linear terms vanish. The remaining inequality compares the
second moments, and hence the variances. Identical laws give equality.
\end{proof}

\subsection{Fixed-weight capacity transfer}

\begin{lemma}\label{lem:capacity-transfer}
Let $G$ be nonempty and let $w(S)>0$ be fixed weights on its independent
sets. Suppose $w(S)t^{|S|}$ gives $\Var|S|<\E|S|$ for every $0<t\le1$.
Then the same holds for weights
\[
 w(S)\binom{N-|S|}{r-|S|},\qquad
 N\ge\alpha(G),\quad1\le r\le N,
\]
where $N,r$ are integers.
\end{lemma}
\begin{proof}
Put $a_k=\sum_{|S|=k}w(S)>0$ for $0\le k\le\alpha(G)$.
At $r=N$ the multiplier is one. Otherwise, apart from a constant,
it is $(r)_k/(N)_k$ on $k\le r$, with consecutive ratios
$(r-k)/(N-k)$ decreasing and at most $r/N$. Thus the conditioned
mean is no larger than the mean under $a_k(r/N)^k$, by the monotone
likelihood ratio comparison. The latter mean is continuous, strictly
increasing from zero, with derivative $\Var_tK/t$. It matches the
positive conditioned mean at some $0<t_*\le r/N<1$.
Relative to the law at $t_*$, the conditioned size law has a
log-concave likelihood ratio and interval support. Equal mass and
mean imply a two-crossing comparison: choose cutpoints $b\le c$
so that $(\nu(k)-\pi(k))(k-b)(k-c)\le0$ for every $k$.
Summing cancels constant and linear terms and yields
$\Var_\nu K\le\Var_\pi K<\E_\pi K=\E_\nu K$.
No order hypothesis $N\ge|V(G)|$ is used.
\end{proof}

\subsection{Corona transfer and quantitative coefficient margins}

\begin{theorem}[A quantitative corona transfer]\label{thm:general-transfer}
Let $H$ be a nonempty finite simple graph of order $h$, let $M\ge h$ be an integer, and let $c>0$ satisfy $\rho=ch/M\le1$. Suppose
\begin{equation}\label{eq:assumption}
 \Var_t K\le\E_t K\qquad(0<t\le c)
\end{equation}
for the independent-set size law of $H$. Define
\[
 B_j=[x^j](1+x)^M I\!\left(H;\frac{cx}{1+x}\right).
\]
Then
\begin{equation}\label{eq:strongtransfer}
 \frac{B_j^2}{B_{j-1}B_{j+1}}>
 \frac{M-j+1}{M-j}\frac{j+1}{j+\rho^2}
 \qquad(1\le j<M).
\end{equation}
In particular, only the nonstrict variance hypothesis is required.
\end{theorem}

\begin{proof}
Put $d=\alpha(H)$, $a_k=c^ki_k(H)$, and $b_k=a_k/\binom Mk$ for $0\le k\le d$, extending $b_k$ by zero beyond $d$. For $2\le r\le M$, let
\[
 T_r=\sum_{k=0}^d\binom rk b_k,\qquad
 \nu_r(k)=\frac{a_k(r)_k/(M)_k}{T_r},\qquad B_r=\binom MrT_r,
\]
where $(u)_k$ is a falling factorial and terms with $k>r$ vanish. Write $\mu=\E_{\nu_r}K$, $v=\Var_{\nu_r}K$, and $m=r-\mu$.

The multipliers $w_k=(r)_k/(M)_k$ have nonincreasing consecutive ratios, bounded above by $r/M$. Thus $\nu_r$ is dominated in likelihood ratio by the size law at activity $cr/M$, and its mean is no larger. This implication follows from the nonpositive covariance of an increasing function and a decreasing function. Its mean is positive and strictly below $d$. Since the hard-core mean increases continuously from zero, with derivative $\Var_tK/t>0$, there is a unique $t\in(0,cr/M]$ at which the size law $p_t$ has mean $\mu$. The ratio $\nu_r(k)/p_t(k)$ is proportional to $w_k(c/t)^k$ and is log-concave on interval support.

Lemma~\ref{lem:relative-lc-variance} and
hypothesis~\eqref{eq:assumption} consequently give
\begin{equation}\label{eq:vmu}
 v\le\mu.
\end{equation}

Counting extensions of independent sets gives
\[
 (k+1)i_{k+1}(H)\le(h-k)i_k(H).
\]
Since $M\ge h$, it follows that $b_{k+1}\le\rho b_k$, and hence
\[
 (k+1)\nu_r(k+1)\le\rho(r-k)\nu_r(k).
\]
Summing yields $\mu\le\rho(r-\mu)=\rho m$.

Fix $r=j+1\ge2$. Binomial identities give
\[
 \frac{T_{r-1}}{T_r}=\frac mr,\qquad
 \frac{T_{r-2}}{T_r}=\frac{m^2-m+v}{r(r-1)},
\]
so
\begin{equation}\label{eq:identity}
 \frac{B_j^2}{B_{j-1}B_{j+1}}
 =\frac{M-j+1}{M-j}\frac{m^2}{m^2-m+v}.
\end{equation}
The denominator is positive: it is $\E[(r-K)(r-K-1)]$, and $\nu_r(0)>0$.

Put $\delta=m-v$. Since $v\le\mu$ and $m=r-\mu$, we have
\[
 \frac{\delta}{m^2}\ge
 \frac{r-2\mu}{(r-\mu)^2}.
\]
The function
\[
 \psi(x)=\frac{r-2x}{(r-x)^2}
\]
is strictly decreasing for $0<x<r$, since
$\psi'(x)=-2x/(r-x)^3<0$. Extension counting gives
$\mu\le r\rho/(1+\rho)$, and therefore
\[
 \frac{\delta}{m^2}\ge
 \psi\!\left(\frac{r\rho}{1+\rho}\right)
 =\frac{1-\rho^2}{r}.
\]
Together with \eqref{eq:identity}, this proves
\begin{equation}\label{eq:rho-squared-transfer}
 \frac{B_j^2}{B_{j-1}B_{j+1}}\ge
 \frac{M-j+1}{M-j}\frac{j+1}{j+\rho^2}.
\end{equation}
In fact, the inequality is always strict. If $H$ has an edge, then
extension counting is strict at $k=1$. It is also strict there when
$M>h$, because $(h-1)/(M-1)<h/M$. Since $r\ge2$ and
$\nu_r(1)>0$, either condition gives
$\mu<r\rho/(1+\rho)$, so strict decrease of $\psi$ proves the
strict version of \eqref{eq:rho-squared-transfer}. In the remaining
case $M=h$ and $H$ has no edges,
$B_j=(1+c)^j\binom Mj$. Its ratio equals
$\frac{M-j+1}{M-j}\frac{j+1}{j}$, which is strictly larger because
$\rho=c>0$. As before, $M=1$ has no internal rank.
\end{proof}

\begin{corollary}[Forest bases with isolated-vertex smoothing]\label{thm:linear-smoothing}
Let $H$ be a nonempty forest of order $h$, let $M\ge h$ be an
integer, and let $0<c\le1$. Put $\rho=ch/M$ and
\[
 B_j=[x^j](1+x)^M I\!\left(H;\frac{cx}{1+x}\right).
\]
Then, for $1\le j<M$,
\[
 \frac{B_j^2}{B_{j-1}B_{j+1}}>
 \frac{M-j+1}{M-j}\frac{j+1}{j+\rho^2}
 \ge\frac{M-j+1}{M-j}>1.
\]
\end{corollary}
\begin{proof}
Lemma~\ref{lem:forest-subpoisson} supplies the variance hypothesis
of Theorem~\ref{thm:general-transfer} throughout $(0,c]$.
\end{proof}

\begin{proof}[Proof of Theorem~\ref{thm:allorder-wc}]
By Lemma~\ref{lem:forest-corona-identity}, write
$F=(H\circ K_1)\sqcup sK_1$, with $h=|V(H)|$ and $m=h+s$.
If $h>0$, Corollary~\ref{thm:linear-smoothing}, with $M=m$ and $c=1$,
gives the stronger inequality
\begin{equation}\label{eq:isolates-gain}
 \frac{i_j(F)^2}{i_{j-1}(F)i_{j+1}(F)}>
 \frac{m-j+1}{m-j}\frac{j+1}{j+(h/m)^2}.
\end{equation}
If $h=0$, then $I(F;x)=(1+x)^m$ and the ratio is exactly
$[(j+1)/j][(m-j+1)/(m-j)]$. Both cases prove
\eqref{eq:main-forest-ratio}; $m\le1$ has no internal ranks.
Multiplication of $i_j(F)$ by $(m-j)!$ multiplies its curvature
ratio by $(m-j)/(m-j+1)$, yielding strict log-concavity of the
normalized sequence. A common positive coefficient tilt leaves
all curvature ratios unchanged.
\end{proof}

\begin{proof}[Proof of Corollary~\ref{cor:forest-clique-coronas}]
For $h=|V(H)|$, choosing the independent set in the base gives
\[
 I(H\circ K_p;x)=(1+px)^h I\!\left(H;\frac{x}{1+px}\right).
\]
Set $y=px$ and apply Corollary~\ref{thm:linear-smoothing} with
$M=h$ and $c=1/p$. Rescaling the variable preserves adjacent
coefficient ratios. This proves~\eqref{eq:clique-ratio}.
The empty forest has no internal ranks.
\end{proof}

For $p>1$, the factor $p^2(j+1)/(p^2j+1)$ in~\eqref{eq:clique-ratio}
is strictly greater than one. As $p$ increases, this factor tends
to $(j+1)/j$, the other factor in the binomial coefficient ratio.
Likewise, \eqref{eq:isolates-gain} records an additional strict
improvement when isolated vertices are adjoined.

\begin{corollary}[A component-sensitive coefficient gap]
\label{cor:component-coefficient-gap}
Let $H$ be a nonempty forest of order $h$, with largest component
order $\ell$. If $M\ge h$ is an integer, $0<c\le1$, and
$B_j=[x^j](1+x)^M I(H;cx/(1+x))$, then
\begin{equation}\label{eq:component-coefficient-gap}
 \frac{B_j^2}{B_{j-1}B_{j+1}}>
 \frac{M-j+1}{M-j}\frac{h\ell+2j+2}{h\ell+2j+1}
 \qquad(1\le j<M).
\end{equation}
In particular, every well-covered forest with independence number $m$
satisfies the coarser universal bound
\begin{equation}\label{eq:universal-finite-margin}
 \frac{i_j(F)^2}{i_{j-1}(F)i_{j+1}(F)}>
 \frac{m-j+1}{m-j}\frac{m^2+2j+2}{m^2+2j+1}
 \qquad(1\le j<m).
\end{equation}
\end{corollary}
\begin{proof}
Use the notation of Theorem~\ref{thm:general-transfer}, with
$r=j+1$, mean $\mu$, variance $v$ under $\nu_r$, and $m=r-\mu$.
Mean matching gives an activity $0<t\le c\le1$. If $V$ is the
variance of the matched hard-core law, Lemma~\ref{lem:relative-lc-variance}
gives $v\le V$. Corollary~\ref{cor:explicit-variance-reserve} implies,
with $D=h\ell$,
\[
 \mu-v\ge\mu-V>\frac{V^2}{D}\ge\frac{v^2}{D}.
\]
The extension bound in Theorem~\ref{thm:general-transfer} gives
$\mu\le\rho m\le m$, where $\rho=ch/M\le1$. Hence
$\delta=m-v>v^2/D$ and
\[
 m^2=v^2+2m\delta-\delta^2<(D+2m)\delta<(D+2r)\delta.
\]
The denominator $m^2-\delta=m^2-m+v$ in~\eqref{eq:identity}
is positive. Therefore that identity gives
\[
 \frac{B_j^2}{B_{j-1}B_{j+1}}
 =\frac{M-j+1}{M-j}\frac1{1-\delta/m^2}
 >\frac{M-j+1}{M-j}\frac{D+2r}{D+2r-1}.
\]
For $F=(H\circ K_1)\sqcup sK_1$, take $c=1$ and
$M=\alpha(F)=h+s$; then $h\ell\le M^2$ gives
\eqref{eq:universal-finite-margin}. If $H$ is empty, the binomial
ratio has the additional factor $(j+1)/j$, which is larger than
$(M^2+2j+2)/(M^2+2j+1)$. There are no internal ranks when $M\le1$.
\end{proof}

The classical star-corona family studied in~\cite{LM2004}
also calibrates the terminal factor in Theorem~\ref{thm:allorder-wc}.

\begin{proposition}[Sharpness at every fixed terminal offset]
\label{prop:terminal-factor-sharpness}
For every integer $q\ge1$,
\[
 \inf_{\substack{F\text{ a well-covered tree}\\\alpha(F)>q}}
 \frac{i_{\alpha(F)-q}(F)^2}
 {i_{\alpha(F)-q-1}(F)i_{\alpha(F)-q+1}(F)}
 =\frac{q+1}{q}.
\]
The infimum is not attained.
\end{proposition}
\begin{proof}
Theorem~\ref{thm:allorder-wc} supplies the strict lower bound.
For $m\ge2$, let $F_m=K_{1,m-1}\circ K_1$, a well-covered tree
with independence number $m$. The classical corona identity gives
\[
 I(F_m;x)=(1+x)(1+2x)^{m-1}+x(1+x)^{m-1}.
\]
For each fixed integer $q\ge0$ and $m>q$, with binomial
coefficients of negative lower index interpreted as zero,
\[
 i_{m-q}(F_m)=
 2^{m-q-1}\left[\binom{m-1}{q}+2\binom{m-1}{q-1}\right]
 +\binom{m-1}{q}
 \sim 2^{m-q-1}\frac{m^q}{q!}.
\]
Applying this asymptotic at $q-1,q,q+1$ shows that the displayed
ratio tends to $(q+1)/q$ as $m\to\infty$.
\end{proof}

\begin{corollary}[An improved universal finite-size margin]
\label{cor:improved-universal-margin}
Let $H$ be a nonempty forest of order $h$, let $M\ge h$ be an
integer, and let $0<c\le1$. Define
\[
 B_j=[x^j](1+x)^M I\!\left(H;\frac{cx}{1+x}\right).
\]
For $1\le j<M$, put
\[
 D_{M,j}=\left(M-\frac{j+1}{2}\right)^2+j+1.
\]
Then
\begin{equation}\label{eq:improved-universal-margin}
 \frac{B_j^2}{B_{j-1}B_{j+1}}>
 \frac{M-j+1}{M-j}\frac{D_{M,j}}{D_{M,j}-1}.
\end{equation}
In particular, every well-covered forest $F$ satisfies this inequality
with $B_j=i_j(F)$ and $M=\alpha(F)$.
\end{corollary}
\begin{proof}
Use the notation and mean-matched hard-core law from the proof of
Theorem~\ref{thm:general-transfer}, with $r=j+1$, tilted-law mean
$\mu$, tilted-law variance $v$, and $m=r-\mu$. Its matched activity
lies in $(0,c]\subseteq(0,1]$. Write $V$ for the variance of the
matched hard-core law. Lemma~\ref{lem:relative-lc-variance} and
Corollary~\ref{cor:mean-sensitive-variance-reserve} imply
\begin{equation}\label{eq:improved-margin-variance-input}
 \mu-v\ge\mu-V\ge\frac{V^2}{(h-\mu)^2}
 \ge\frac{v^2}{(M-\mu)^2}.
\end{equation}
Also, extension counting gives
$0<\mu\le\rho m\le m$, where $\rho=ch/M\le1$, so
$\mu\le r/2$.

For $0<x\le r/2$, let $y=y(x)>0$ be the unique solution of
\[
 x-y=\frac{y^2}{(M-x)^2}.
\]
The right-hand side of $x=y+y^2/(M-x)^2$ is strictly increasing
in $y\ge0$, so \eqref{eq:improved-margin-variance-input} gives
$v\le y(\mu)$. Set
\[
 f(x)=\frac{r-x-y(x)}{(r-x)^2}.
\]
We claim that $f$ is strictly decreasing. Put $B=M-x$ and
$m_x=r-x$. Implicit differentiation gives
\[
 1-y'(x)=\frac{2y(B+y)}{B(B^2+2y)},\qquad
 f'(x)=\frac{2y}{m_x^3}
 \left(\frac{m_x(B+y)}{B(B^2+2y)}-1\right).
\]
Here $m_x\le B$ because $r\le M$, and
$B\ge M-r/2\ge1$ because $M\ge2$. Consequently,
\[
 \frac{m_x(B+y)}{B(B^2+2y)}
 \le\frac{B+y}{B^2+2y}<1,
\]
where the strict inequality follows from $B^2-B+y>0$.
This proves the claim.

Put $L=M-r/2$ and $\delta=m-v$. Evaluating $f$ at $r/2$ yields
\[
 \frac{\delta}{m^2}\ge f(\mu)\ge f(r/2)
 =\frac{4}{\left(L+\sqrt{L^2+2r}\right)^2}
 >\frac1{L^2+r}=\frac1{D_{M,j}}.
\]
For the strict inequality, square the positive quantities to obtain
$L\sqrt{L^2+2r}<L^2+r$, and hence
$\left(L+\sqrt{L^2+2r}\right)^2<4(L^2+r)$.
Since $m^2-\delta=m^2-m+v>0$, identity~\eqref{eq:identity} gives
\[
 \frac{B_j^2}{B_{j-1}B_{j+1}}
 =\frac{M-j+1}{M-j}\frac1{1-\delta/m^2}
 >\frac{M-j+1}{M-j}\frac{D_{M,j}}{D_{M,j}-1}.
\]
For a well-covered forest with a nonempty base, use
Lemma~\ref{lem:forest-corona-identity}. If its base is empty, then
$i_j(F)=\binom Mj$ and its ratio is
$\frac{M-j+1}{M-j}\frac{j+1}{j}$, which is strictly larger because
$D_{M,j}>j+1$. There are no internal ranks when $M\le1$.
\end{proof}

\subsection{One less binomial factor}\label{subsec:one-less-smoothing}
The preceding forest transform uses a binomial exponent at least the
base order. If the base has an edge, one factor can be removed while
retaining ordinary strict log-concavity. The proof uses the same
mean-matching comparison, together with a sharper extension bound
for independent sets of positive size.

\begin{theorem}[Removing one binomial factor]\label{thm:one-less-smoothing}
Let $H$ be any forest containing at least one edge, and put $h=|V(H)|$.
Then
\[
 Q(x)=(1+x)^{h-1}I\left(H;\frac{x}{1+x}\right)
     =\frac{I(H\circ K_1;x)}{1+x}
\]
is a polynomial of degree $h-1$ with strictly positive coefficients,
and it is strictly log-concave at every internal index.
\end{theorem}
\begin{proof}
First assume that $H$ has no isolated vertices and put $M=h-1$.
Because $H$ has an edge, $\alpha(H)\le M$, so the displayed
expression is a polynomial with positive coefficients at every index
$0,\ldots,M$. The case $M=1$ has no internal index.

Write $a_k=i_k(H)$ and $B_j=[x^j]Q(x)$. For $0\le r\le M$, define
\[
 T_r=\sum_k a_k\frac{(r)_k}{(M)_k},\qquad
 \nu_r(k)=\frac{a_k(r)_k/(M)_k}{T_r}.
\]
Terms with $k>r$ vanish, and coefficient extraction gives
$B_r=\binom MrT_r$. Now fix $2\le r\le M$ and let
$\mu=\mathbb E_{\nu_r}K$,
$v=\operatorname{Var}_{\nu_r}K$, and $m=r-\mu$.
As in the proof of Theorem~\ref{thm:general-transfer}, the multiplier
$(r)_k/(M)_k$ has nonincreasing consecutive ratios, at most $r/M$.
Thus $\nu_r$ has mean no greater than the hard-core size law of $H$
at common activity $r/M\le1$. Matching that mean to an activity
$t\in(0,r/M]$, the likelihood ratio to the corresponding size law
is log-concave. Lemma~\ref{lem:relative-lc-variance} and the forest variance bound
in Lemma~\ref{lem:forest-subpoisson} therefore imply
\[
 v\le\mu.
\]
This argument needs $M\ge\alpha(H)$, not $M\ge h$.

For every nonempty independent set $S$ of $H$, at least one vertex
outside $S$ is adjacent to $S$, because $H$ has no isolated vertices.
Therefore the number of possible one-vertex extensions is at most
$h-|S|-1=M-|S|$, and double counting gives
\[
 (k+1)a_{k+1}\le(M-k)a_k\quad(k\ge1),
 \qquad a_1=M+1.
\]
After multiplication by the falling-factorial weights, this becomes
\[
 (k+1)\nu_r(k+1)\le(r-k)\nu_r(k)\quad(k\ge1),
 \qquad \nu_r(1)=\frac{r(M+1)}M\nu_r(0).
\]
The endpoint $k=M$ is interpreted directly as $0\le0$; no ratio
with denominator $M-k$ is evaluated there. Summing yields
\[
 2\mu\le r+\frac rM\nu_r(0).
\]
Put $\varepsilon=\nu_r(0)/M$. Since
$1\ge\nu_r(0)+\nu_r(1)$,
\[
 0<\varepsilon\le\frac1{M+r(M+1)}
    <\frac1{M(r+1)}\le\frac16.
\]
Consequently
\[
 m\ge\frac r2(1-\varepsilon)>\frac{5r}{12},
 \qquad v-m\le\mu-m=2\mu-r\le r\varepsilon.
\]

For $j=r-1$, the exact coefficient identity is
\[
 \frac{B_j^2}{B_{j-1}B_{j+1}}
 =\frac{M-r+2}{M-r+1}
       \frac{m^2}{m^2-m+v}.
\]
Its denominator is positive, since it is
$\mathbb E_{\nu_r}[(r-K)(r-K-1)]$ and $r\ge2$, $\nu_r(0)>0$.
The displayed ratio exceeds one precisely when
\[
 m^2>(M-r+1)(v-m).
\]
If $v-m\le0$ this is immediate. Otherwise,
\[
 (M-r+1)(v-m)\le(M-r+1)r\varepsilon
 <\frac r{r+1}
 <\frac{25r^2}{144}<m^2.
\]
The middle strict inequality follows from
$25r(r+1)>144$, valid for every integer $r\ge2$.
This proves strict LC at all internal indices.

Finally, write a general $H$ as $H_0\sqcup zK_1$, where $H_0$
has no isolated vertices and still has an edge. If $h_0=|V(H_0)|$,
\[
 Q(x)=(1+2x)^z(1+x)^{h_0-1}
           I\left(H_0;\frac{x}{1+x}\right).
\]
The second factor is strictly LC by the preceding proof, and
convolution with each positive linear factor preserves strict LC.
For completeness, if $c_j=b_j+2b_{j-1}$ and $b$ is zero outside its
positive support, then
\[
 c_j^2-c_{j-1}c_{j+1}
 =\Delta_j(b)+4\Delta_{j-1}(b)
       +2(b_jb_{j-1}-b_{j-2}b_{j+1}),
\]
where $\Delta_j(b)=b_j^2-b_{j-1}b_{j+1}$. Decreasing consecutive coefficient ratios
make the last term nonnegative, and an internal or endpoint determinant
makes the sum strictly positive at every new internal index.
This also verifies the claimed degree and positive support.
\end{proof}

The conclusion concerns a polynomial quotient. It does not assert strict
log-concavity after deleting an arbitrary private leaf of a corona:
that deletion contributes an additional shifted positive term, whose
mixed determinant requires a separate argument. Nor is the factorial
factor of Theorem~\ref{thm:allorder-wc} asserted for this quotient.
The theorem removes one factor; it does not settle every smaller
binomial exponent down to $\alpha(H)$.

\subsection{The exact binomial ultra-log-concavity prefix}
\begin{lemma}[Adding an isolated vertex preserves an initial ULC prefix]
\label{lem:ulc-prefix-isolate}
Let $a_0,\ldots,a_n>0$, and write $r_k=a_k/\binom nk$.
Put $b_j=a_j+a_{j-1}$, where coefficients outside $\{0,\ldots,n\}$
are zero, and $c_j=b_j/\binom{n+1}{j}$.
For $1\le j\le n$, suppose that $(r_k)$ is log-concave at every
index among $j-1,j$ that lies in $\{1,\ldots,n-1\}$. Then
\[
 c_j^2-c_{j-1}c_{j+1}\ge
 \frac{(r_j-r_{j-1})^2}{(n+1)^2}\ge0.
\]
Consequently, multiplication by $1+x$ preserves any initial prefix
of binomial ultra-log-concavity.
\end{lemma}
\begin{proof}
Set $A=n+1-j$, $B=j$, $x=r_{j-1}$, and $y=r_j$. Then
$(n+1)c_j=Ay+Bx$. Log-concavity gives
$r_{j-2}\le x^2/y$ and $r_{j+1}\le y^2/x$ whenever needed.
At $j=1$ the first missing term has coefficient $B-1=0$;
at $j=n$ the second missing term has coefficient $A-1=0$.
Thus
\begin{align*}
 (n+1)^2(c_j^2-c_{j-1}c_{j+1})
 &\ge (Ay+Bx)^2\\
 &\quad-\bigl((A+1)x+(B-1)x^2/y\bigr)
           \bigl((A-1)y^2/x+(B+1)y\bigr)\\
 &=(y-x)^2.
\end{align*}
The argument also covers the newly exposed last internal index
when the original degree is smaller than the prescribed prefix.
\end{proof}

\begin{theorem}[The exact universal ULC prefix for well-covered forests]
\label{thm:exact-ulc-prefix}
Let $F$ be a well-covered forest and $m=\alpha(F)$. Then
\[
 \left(\frac{i_j(F)}{\binom mj}\right)^2
 \ge\frac{i_{j-1}(F)}{\binom m{j-1}}
       \frac{i_{j+1}(F)}{\binom m{j+1}}
 \qquad\bigl(1\le j\le\min\{3,m-1\}\bigr).
\]
At each such index the inequality is strict unless all components
of $F$ are single vertices or all components are single edges.
For every fixed $j\ge4$, infinitely many well-covered trees violate
this inequality at index $j$. Hence the universal fixed ULC indices
for well-covered forests are exactly $1,2,3$.
\end{theorem}
\begin{proof}
First let $F=H\circ K_1$ with $h=|V(H)|$ and $H$ a forest.
Write $e=|E(H)|$, $S=\sum_v\binom{\deg(v)}2$, and
\[
 b_k=\frac{i_k(H)}{\binom hk},\qquad
 T_j=\frac{i_j(F)}{\binom hj}=\sum_k\binom jk b_k.
\]
Orders $h\le1$ have no internal ranks. For the remaining orders,
each expression below is used only when its denominator exists. Put
$u=e/\binom h2$ and $v=S/\binom h3$.
Since $H$ has no triangle,
\[
 b_2=1-u,\qquad b_3=1-3u+v,
 \qquad T_1=2,\quad T_2=4-u,\quad T_3=8-6u+v.
\]
Moreover,
\[
 2S=\sum_{ab\in E(H)}(\deg(a)+\deg(b)-2)\le e(h-2),
 \qquad 0\le v\le\tfrac32u.
\]
Indeed, the neighbor sets of the ends of an edge have no common
vertex, so $\deg(a)+\deg(b)\le h$.
It follows that
\[
 T_1^2-T_0T_2=u,\qquad
 T_2^2-T_1T_3=4u+u^2-2v\ge u+u^2.
\]
For $h\ge4$, choose a uniformly random four-vertex subset and let
$X$ count its induced edges. Every induced subgraph is a forest,
so $0\le X\le3$ and
\[
 1-b_4=\Pr(X\ge1)\ge\frac{\mathbb EX}{3}=2u.
\]
Thus
\[
 T_4=15-18u+4v+b_4\le16-20u+4v.
\]
Since $T_2=4-u>0$,
\begin{align*}
 T_3^2-T_2T_4
 &\ge(8-6u+v)^2-(4-u)(16-20u+4v)\\
 &=(4u-v)^2\ge\tfrac{25}{4}u^2.
\end{align*}
These inequalities are strict when $e>0$. When $e=0$,
$I(F;x)=(1+2x)^h$, so all normalized determinants vanish.

For a general well-covered forest, write
$F=(H\circ K_1)\sqcup sK_1$ and apply the preceding lemma $s$ times.
If $H$ is empty, $I(F;x)=(1+x)^s$ and equality holds.
If $h>0$ and $s>0$, strictness follows from that lemma: before
and after each added isolated vertex, the normalized coefficients
are strictly increasing. Explicitly, at total degree $n\ge h$,
\[
 \frac{[x^j](1+x)^nI(H;x/(1+x))}{\binom nj}
 =\sum_k\binom jk\frac{i_k(H)}{\binom nk},
\]
whose consecutive difference is positive because its $k=1$ term
contributes $h/n>0$. This also proves the asserted equality cases.

Finally, for the star corona $F_h=K_{1,h-1}\circ K_1$,
\[
 I(F_h;x)=(1+x)(1+2x)^{h-1}+x(1+x)^{h-1},\qquad
 T_j=\frac{2^{j-1}(2h-j)+j}{h}.
\]
Direct expansion gives
\[
 T_j^2-T_{j-1}T_{j+1}
 =\frac{2^{2j-2}+1+
  2^{j-2}\bigl(2(3-j)h+j^2-5\bigr)}{h^2}.
\]
For example, at $h=14$ the coefficient triple at ranks $3,4,5$ is
$(2678,14014,53339)$, and the normalized fourth-rank determinant is
$-3/196$. For every fixed $j\ge4$, the displayed numerator is
affine in $h$ with negative slope. It is therefore negative for all sufficiently
large integers $h>j$.
\end{proof}

\begin{remark}
The positive prefix proof for $H\circ K_1$ uses only that every
induced subgraph of $H$ on at most four vertices is a forest.
It therefore applies, with the same strictness, to every base
of girth at least five. Adding isolated vertices remains allowed.
\end{remark}

\subsection{An application to corona \texorpdfstring{$h$}{h}-polynomials}
\label{subsec:corona-h-polynomials}
Write $h_G(t)$ for the $h$-polynomial of the independence complex of
$G$, so that
\[
 h_G(t)=(1-t)^{\alpha(G)}I\!\left(G;\frac{t}{1-t}\right).
\]
Hoang, Pham, and Trung~\cite[Propositions 7.5--7.6 and
Corollary 7.7, pp.~2550343-17--2550343-18]{HPTCorona2026}
give the uniform-attachment corona $h$-transform and its clique specialization.
In particular, for $n=|V(G)|$ and every integer $p\ge1$,
\begin{equation}\label{eq:known-corona-h-transform}
 h_{G\circ K_p}(t)
 =\sum_{k=0}^{\alpha(G)}i_k(G)t^k[1+(p-1)t]^{n-k}.
\end{equation}
This polynomial identity follows by substituting the classical
clique-corona identity into the defining $f/h$ relation; the displayed
sum also specifies the value when $1+(p-1)t=0$. Consequently,
\[
 h_{G\circ K_1}(t)=I(G;t),\qquad
 h_{G\circ K_p}(t)=I(G\circ K_{p-1};t)\quad(p\ge2).
\]
Both identities are direct applications of the published transform.
Combining them with our coefficient bound gives the following consequence.

\begin{corollary}[Quantitative $h$-log-concavity for forest clique coronas]
\label{cor:forest-corona-h}
Let $F$ be a nonempty forest of order $n$, and let $p\ge2$ be an
integer. Write $h_{F\circ K_p}(t)=\sum_{j=0}^n h_jt^j$. Then
\[
 \frac{h_j^2}{h_{j-1}h_{j+1}}>
 \frac{n-j+1}{n-j}
 \frac{(p-1)^2(j+1)}{(p-1)^2j+1}
 \qquad(1\le j<n).
\]
In particular, both $(h_j)_{j=0}^n$ and
$((n-j)!h_j)_{j=0}^n$ are strictly log-concave.
\end{corollary}
\begin{proof}
Use $h_{F\circ K_p}=I(F\circ K_{p-1})$ and
Corollary~\ref{cor:forest-clique-coronas}, with attachment size $p-1$.
The second factor is at least one. All coefficients are positive,
and the degree is $n$, by \eqref{eq:known-corona-h-transform}.
\end{proof}

The restriction $p\ge2$ is essential. If a well-covered forest is
$W=(F\circ K_1)\sqcup sK_1$, its $h$-polynomial is $h_W=I(F)$:
each isolated vertex has $h$-polynomial $1$. Known forest bases with
non-log-concave independence polynomials~\cite{KL26,GalvinFailures}
therefore give well-covered forests with non-log-concave $h$-polynomials,
even though Theorem~\ref{thm:allorder-wc} proves strict log-concavity
of their independence polynomials. The two coefficient sequences must
be distinguished.

\section{Marked roots and fixed boundary interactions}\label{sec:marked-boundaries}
\subsection{The square reserve}
The single-root estimates in this section retain quantitative
variance reserves after marking the root. Clique assembly gives the
adjacent OR-event law, and repulsive soft-edge coupling gives the
absent--absent event law. Lemma~\ref{lem:capacity-transfer} supplies
their canonical versions. All marks remain fixed while the common
activity varies. A size-dependent boundary weight requires a separate
comparison.

Let $W$ be the nonnegative inverse of $w\mapsto we^w$, and put
$\omega=W(e^{-1})$, $\Theta=\tfrac12\exp(3-2\omega)$.
\begin{lemma}[Sharp one-message envelope]\label{lem:sharp-message-envelope}
For $0<t\le1$, set $a=t/(1+t)$ and $w=W(t/e)$. Then
\[
 \max_{0<q\le a}q\log\frac{t(1-q)}q=w.
\]
The maximum is attained uniquely at $q=w/(1+w)$. Moreover,
\[
 \beta(t):=\frac{1+t}{t}W(t/e)\le2\omega.
\]
\end{lemma}
\begin{proof}
For $f(q)=q\log(t(1-q)/q)$,
\[
 f'(q)=\log\frac{t(1-q)}q-\frac1{1-q},\qquad
 f''(q)=-\frac1{q(1-q)^2}<0.
\]
Writing $x=q/(1-q)$, the critical-point equation is
$\log(t/x)=1+x$, or $xe^x=t/e$. Hence $x=w$, and
$f(q)=w$ there. Since $w<t$, this critical point lies in $(0,a)$.
The endpoint values are $\lim_{q\downarrow0}f(q)=f(a)=0$.

Implicit differentiation of $we^w=t/e$ gives
$w'=w/[t(1+w)]$. Consequently
\[
 \beta'(t)=\frac{w(t-w)}{t^2(1+w)}>0,
\]
so $\beta(t)\le\beta(1)=2\omega$.
\end{proof}

Fix a common activity $0<t\le1$, and put $a=t/(1+t)$.
For a rooted tree $(T,o)$, assign the root the fixed multiplier
$0<\theta\le\Theta$, with every other multiplier equal to one. Let
$q=\Pr(o\in S)$, let
\[
 d=\mathbb E[|S|\mid o\in S]-\mathbb E[|S|\mid o\notin S],
 \qquad g=\mathbb E|S|-\operatorname{Var}|S|.
\]
The multiplier $\theta$ is fixed when varying $t$.

\begin{theorem}[Improved marked square reserve]\label{thm:marked-square-reserve}
Let $(T,o)$ be a nonempty rooted tree. Give $o$ a fixed multiplier
$0<\theta\le\Theta$, give every other vertex multiplier one, and use
common activity $0<t\le1$. With $q,d,g$ as above,
\[
 g\ge q^2d^2.
\]
Equality holds for a singleton, and the inequality is strict whenever
$T$ has more than one vertex. The endpoint $\theta=\Theta,t=1$ is
included.
\end{theorem}
\begin{proof}
The singleton case is Bernoulli: $d=1$ and $g=q^2$.
Otherwise let the unweighted child trees have data $q_i,d_i,g_i$,
and let $g_i^-$ be their root-deleted deficits. Put
\[
 L_i=\log\frac{t(1-q_i)}{q_i},\quad
 s=1-q,\quad S=\sum_iq_i>0,\quad u=\sum_iq_id_i.
\]
Lemma~\ref{lem:forest-subpoisson} and Cauchy--Schwarz on the grandchildren
give
\[
 g_i\ge aq_id_i^2,\qquad
 g_i^-\ge a\frac{(1-d_i)^2}{L_i}.
\]
When $L_i=0$, the child is a singleton and $d_i=1$; its last term is
interpreted as zero subject to this exact constraint. The root odds
and the conditional-variance identity give
\[
 \frac qs=\theta t\prod_i(1-q_i),\quad d=1-u,\qquad
 g=s\sum_i g_i+q\sum_i g_i^-+q-qs(1-u)^2.
\]
It is therefore enough to prove positivity of
\[
 Q=as\sum_iq_id_i^2+
 aq\sum_{i:L_i>0}\frac{(1-d_i)^2}{L_i}+q-q(1-u)^2.
\]
Set
\[
 m=\sum_i\frac{q_iq}{q+s q_iL_i},\qquad
 V=\frac{S-m}{as}.
\]
Completing squares and applying weighted Cauchy--Schwarz yields
\[
 as\sum_iq_id_i^2+
 aq\sum_{i:L_i>0}\frac{(1-d_i)^2}{L_i}
 \ge as m+\frac{(u-m)^2}{V}
\]
if $V>0$. If $V=0$, then $u=m=S$ and the left side equals $as m$.
For completeness, for $L_i>0$ take
$A_i=a(sq_i+q/L_i)$ and $B_i=aq/L_i$. The identities are
\[
 \sum_i\frac{q_iB_i}{A_i}=m,\quad
 \sum_i\frac{q_i^2}{A_i}=V,\quad
 \sum_i\left(B_i-\frac{B_i^2}{A_i}\right)=as m.
\]
An index with $L_i=0$ contributes $q_i,0,asq_i$, respectively.

It suffices to show
\[
 F=(q+as m)(1-qV)-q(1-m)^2>0.
\]
Indeed this implies $1-qV>0$, and the minimum of the resulting
quadratic in $u$ is $F/(1-qV)>0$ when $V>0$; for $V=0$ the
constrained value is $F$.
Direct expansion gives
\[
 F=-\frac{q^2S}{as}
   +m\left[as+q(2-S)+\frac{q^2}{as}\right].
\]
By Lemma~\ref{lem:sharp-message-envelope},
\[
 q_iL_i\le w:=W(t/e),\qquad
 m\ge\frac{Sq}{q+sw}.
\]
Writing $x=q/(as)$, the root odds give the strict inequality
\[
 x=\theta(1+t)\prod_i(1-q_i)
   <\theta(1+t)e^{-S},
\]
because there is at least one child and each $q_i>0$.
The coefficient of $m$, divided by $as$, is
$1+x(2-S)+x^2$. For $S\le2$ it is positive. For $S>2$, it exceeds
\[
 1-2\Theta e^{-S}(S-2)
 \ge1-2\Theta e^{-3}=1-e^{-2\omega}>0.
\]
The lower bound for $m$ can therefore be substituted, yielding
\[
 F\ge\frac{Sq}{q+sw}
 \left[as+q\left(2-S-\frac wa\right)\right].
\]
Put $\delta=2-w/a=2-\beta(t)>0$. If $S\le\delta$, the bracket is
positive. If $S>\delta$, its value divided by $as$ is
\begin{align*}
 1-x(S-\delta)
 &>1-\theta(1+t)e^{-S}(S-\delta)\\
 &\ge1-\theta(1+t)e^{-\delta-1}\\
 &=1-\theta(1+t)e^{-3+\beta(t)}\\
 &\ge1-2\Theta e^{-3+2\omega}=0.
\end{align*}
This uses $\max_{S\ge\delta}(S-\delta)e^{-S}=e^{-\delta-1}$
and $\beta(t)\le2\omega$. The first inequality remains strict even
at $\theta=\Theta,t=1$. Hence $F>0$, so $g-q^2d^2\ge Q>0$.
\end{proof}

\paragraph{Elementary bounds on the constant.}
One has $5<\Theta<6$. To see the lower bound, $e^{13/10}>10/3$
follows from the first four exponential-series terms, so
$(3/10)e^{3/10}>e^{-1}$ and $\omega<3/10$. Also
\[
 e^{12/5}>\sum_{k=0}^5\frac{(12/5)^k}{k!}
 =\frac{166093}{15625}>10.
\]
Therefore $\Theta>e^{12/5}/2>5$.
For the upper bound use $e<11/4$ and, for $0<x<3$,
$e^x\le1+x+x^2/[2(1-x/3)]$. At $x=27/100$ these give
\[
 \frac{27}{100}e^{127/100}
 <\frac{27}{100}\frac{11}{4}
   \left(1+\frac{27}{100}
       +\frac{(27/100)^2}{2(1-9/100)}\right)<1,
\]
so $\omega>27/100$. Consequently
\[
 2\Theta<e^{123/50}
 <\frac{121}{16}\left(1+\frac{23}{50}
       +\frac{(23/50)^2}{2(1-23/150)}\right)
 =\frac{2435609}{203200}<12.
\]

\begin{corollary}[Unrestricted mark at a low-degree root]\label{cor:low-degree-mark}
Let $(T,o)$ be a nonempty rooted tree with $\deg_T(o)\le4$.
Give the root an arbitrary fixed multiplier $\theta>0$, give all other
vertices multiplier one, and use common activity $0<t\le1$.
Write $q$ for root occupation probability, $d$ for the occupied-minus-
unoccupied conditional mean size, and $g=\mathbb E K-\operatorname{Var}K$.
Then $g\ge q^2d^2$, with strict inequality unless $T$ is a singleton.
\end{corollary}
\begin{proof}
The singleton assertion is the Bernoulli identity. In the nonsingleton
case use the proof and notation of Theorem~\ref{thm:marked-square-reserve}, but impose no upper bound on $\theta$. Put
$a=t/(1+t)$, $s=1-q$, $S=\sum_iq_i$, and
\[
 m=\sum_i\frac{q_iq}{q+s q_iL_i},\qquad
 V=\frac{S-m}{as},\qquad x=\frac{q}{as},\qquad
 Z=\frac{\sum_iq_i^2L_i}{aS}.
\]
Here $L_i=\log[t(1-q_i)/q_i]\ge0$, the number of children is at most
four, and $0<q_i\le a\le1/2$. In particular $0<S\le4a\le2$.
Since $y\mapsto1/(x+y)$ is convex, weighted Jensen gives
\[
 m=\sum_iq_i\frac{x}{x+q_iL_i/a}\ge\frac{Sx}{x+Z}.
\]
The elementary bound $\log y\le y-1$ yields
\[
 \sum_iq_i^2L_i
 \le\frac{aS-\sum_iq_i^2}{1-a}
 \le\frac{aS-S^2/4}{1-a}
 \le aS(2-S).
\]
For the last inequality, subtract its left-hand numerator from
$aS(2-S)(1-a)$ and divide by $aS>0$. The result is
\[
 1-2a+S\left(a-1+\frac1{4a}\right)
 =1-2a+S\frac{(2a-1)^2}{4a}\ge0.
\]
Thus $Z\le2-S$. The completed-square determinant is
\[
 \frac{F}{as}=-x^2S+m[1+x(2-S)+x^2].
\]
Its coefficient of $m$ is positive because $S\le2$, so
\[
 \frac{F}{as}\ge\frac{Sx}{x+Z}
       \bigl[1+x(2-S-Z)\bigr]>0.
\]
The completed-square argument in Theorem~\ref{thm:marked-square-reserve}
now gives $g-q^2d^2\ge Q>0$: if $V>0$, the minimum is
$F/(1-qV)>0$, while for $V=0$ the constrained value is $F$. No upper bound on the marked root's odds was
used. The cases $L_i=0$ retain the exact constraint $d_i=1$ from the
inherited proof and are included in the displayed formulas.
\end{proof}

\paragraph{Sharpness of the degree threshold.}
For the rooted star $K_{1,D}$ with its center marked by $\theta$, at
$t=1$ one has $q=\theta/(\theta+2^D)$ and $d=1-D/2$. Direct
conditioning on the center gives
\[
 g-q^2d^2=\frac D4\bigl(1-(D-3)q\bigr).
\]
At $D=5$ and $\theta=33$, this is $-1/52<0$. Thus four is the
largest universal root-degree threshold for the unrestricted-mark
square reserve. This example concerns the reserve, not full strict
LC of a one-cut forest.

\subsection{A stronger reserve for marks through 15/14}
The following extends the unweighted reserve $aqd^2$, rather than only
the weaker marked reserve $q^2d^2$. All multipliers remain fixed as
the common activity varies.

\begin{theorem}[Strong marked-root reserve]\label{thm:strong-marked-reserve}
Let $(T,o)$ be a nonempty rooted tree. Give $o$ a fixed multiplier
$0<\theta\le15/14$ and every other vertex multiplier one. At common
activity $0<t\le1$, put
\[
 a=\frac{t}{1+t},\qquad h=\frac{\theta t}{1+\theta t},\qquad
 q=\Pr(o\in S),\qquad
 d=\mathbb E[K\mid o\in S]-\mathbb E[K\mid o\notin S],
\]
where $K=|S|$, and put $g=\mathbb EK-\operatorname{Var}K$. Then
\[
 g\ge hq d^2.
\]
Equality holds for a singleton, and the inequality is strict for every
larger tree. Both endpoints $\theta=15/14$ and $t=1$ are included.
Independent unweighted forest components may be added: the same root
reserve continues to hold, and any nonempty added component contributes
a strictly positive deficit.
\end{theorem}
\begin{proof}
For a singleton, $q=h$, $d=1$ and $g=h^2$. Otherwise write the
unweighted child-root data as $q_i,d_i,g_i$, with root-deleted deficits
$g_i^-$. Put
\[
 s=1-q,\quad S=\sum_iq_i>0,\quad u=\sum_iq_id_i,\quad
 L_i=\log\frac{t(1-q_i)}{q_i}.
\]
Lemma~\ref{lem:forest-subpoisson} and its deletion bound give
\[
 0<q_i\le a,\qquad g_i\ge aq_id_i^2,\qquad
 g_i^-\ge a\frac{(1-d_i)^2}{L_i}.
\]
When $L_i=0$, the child is a singleton and $d_i=1$; the quotient term
is zero subject to this exact constraint. The root odds and conditional
mixture identity are
\[
 \frac qs=\theta t\prod_i(1-q_i),\quad d=1-u,\qquad
 g=s\sum_i g_i+q\sum_i g_i^-+q-qs(1-u)^2.
\]
Thus, writing $C=s+h$, it is enough to prove positivity of
\[
 Q=as\sum_iq_id_i^2+aq\sum_{i:L_i>0}\frac{(1-d_i)^2}{L_i}
       +q-qC(1-u)^2.
\]

Set
\[
 m=\sum_i\frac{q_iq}{q+s q_iL_i},\qquad V=\frac{S-m}{as}.
\]
The same square completion used in Lemma~\ref{lem:rooted-scalar-positivity} gives
\[
 as\sum_iq_id_i^2+aq\sum_{i:L_i>0}\frac{(1-d_i)^2}{L_i}
 \ge as m+\frac{(u-m)^2}{V}
\]
when $V>0$. When $V=0$, $u=m=S$ and the left side is $as m$.
For completeness, on an index with $L_i>0$, take
$A_i=a(sq_i+q/L_i)$ and $B_i=aq/L_i$. Then
\[
 \sum_i\frac{q_iB_i}{A_i}=m,\quad
 \sum_i\frac{q_i^2}{A_i}=V,\quad
 \sum_i\left(B_i-\frac{B_i^2}{A_i}\right)=as m.
\]
An index with $L_i=0$ contributes $q_i,0,asq_i$, respectively.
Define the determinant
\[
 F=(q+as m)(1-qCV)-qC(1-m)^2.
\]
If $F>0$, then $1-qCV>0$. Completing the remaining square in $u$
proves $Q\ge F/(1-qCV)>0$ when $V>0$; when $V=0$, $Q=F>0$.
Direct expansion gives
\[
 F=q(q-h)-\frac{q^2CS}{as}
    +m\left[as+qC\left(2+\frac q{as}-S\right)\right].       \tag{S1}
\]

Use the sharp envelope $w=W(t/e)$, from Lemma~\ref{lem:sharp-message-envelope}: $q_iL_i\le w$, and
\[
 \beta:=w/a\le2\omega<3/5,\qquad \omega=W(1/e)<3/10.
\]
Consequently
\[
 m\ge\frac{Sq}{q+sw},\qquad \delta:=2-\beta>7/5.
\]
Put
\[
 L=\sum_i-\log(1-q_i),\qquad c=\frac a{\log(1+t)},\qquad
 R=\frac{qC}{as}.
\]
Since $-\log(1-x)/x$ is increasing and $q_i\le a$, one has
$cL\le S\le L$. The marked root odds give
\[
 h-q=hs(1-e^{-L}),\qquad
 R=\theta(1+t)(s+h)e^{-L}<4e^{-L}.                  \tag{S2}
\]
For the last inequality, $h\le15/29$ and
$\theta(1+t)(s+h)\le(15/14)\,2\,(1+15/29)=660/203<4$.

The coefficient of $m$ in \textup{(S1)} is positive. Divided by $as$, it is
$1+R(2+q/(as)-S)$, positive immediately if $S\le2$. If $S>2$,
\textup{(S2)}, $L\ge S$, and $(S-2)e^{-S}\le e^{-3}$ show that it is
strictly greater than $1-4e^{-3}>1/2$, using $e>2$.
We may therefore substitute the lower bound on $m$ into \textup{(S1)}, obtaining
\[
 F\ge\frac q{q+sw}G,
\]
\[
 G=(q-h)(q+sw)+S\left[as+qC\left(2-S-\frac wa\right)\right].
\]
Using \textup{(S2)} gives the normalized expression
\[
 \frac G{as}=S-\frac ha(1-e^{-L})(q+sw)
                   +RS(\delta-S).                         \tag{S3}
\]
The following uniform bounds are sufficient:
\[
 \frac ha\le\frac{15}{14},\qquad
 q+sw\le h+(1-h)w<\frac{96}{145},\qquad
 \frac1c\le2\log2<\frac75.                                 \tag{S4}
\]
Indeed, $h/a=\theta(1+t)/(1+\theta t)\le\max\{1,\theta\}$;
$q\le h\le15/29$ and $w<3/10$ give the middle bound; the last
uses monotonicity of $-\log(1-a)/a$ on $0<a\le1/2$.

If $S\le\delta$, the last term of \textup{(S3)} is nonnegative. Using
$1-e^{-L}\le L\le S/c$ and \textup{(S4)},
\[
 \frac G{as}>S\left(1-\frac{15}{14}\frac{96}{145}\frac75\right)
             =\frac S{145}>0.
\]
If $S>\delta$, then
\[
 \frac{(h/a)(1-e^{-L})(q+sw)}S
 <\frac{15}{14}\frac{96}{145}\frac57=\frac{720}{1421}.
\]
Also, by \textup{(S2)} and $L\ge S$,
\[
 R(S-\delta)<4e^{-S}(S-\delta)
 \le4e^{-\delta-1}<4e^{-12/5}<\frac25.
\]
The last inequality uses $e^{12/5}>10$. Thus \textup{(S3)} yields
\[
 \frac G{as}>S\left(1-\frac{720}{1421}-\frac25\right)
             =\frac{663}{7105}S>0.
\]
The elementary constants $\omega<3/10$, $\log2<7/10$ and
$e^{12/5}>10$ were proved above from positive exponential-series partial sums. Therefore $F>0$, and hence
$g-hqd^2\ge Q>0$ for every nonsingleton tree. Adding independent
unweighted components adds their strictly positive deficits without
altering $q$ or $d$.
\end{proof}

\paragraph{Scope.}
The multiplier range is sufficient, not asserted optimal. This theorem
allows one marked root, with unweighted vertices elsewhere in that
rooted piece. It does not itself establish a general two-mark variance
statement, a reference-law comparison, or any new coefficient BL theorem.

\subsection{A stronger logarithmic deletion-deficit floor}
Throughout this section, $0<t\le1$ is the common unweighted activity,
$a=t/(1+t)$, and $g_G=\mathbb E K_G-\operatorname{Var}K_G$.
Lemma~\ref{lem:forest-subpoisson} supplies $g_G\ge0$ and,
for any chosen root of a tree, the rooted reserve $g_G\ge aq d^2$.

\begin{lemma}[Root-probability and logarithmic floors]\label{lem:four-root-floor}
For a nonempty rooted unweighted tree, let $q$ be its root occupation
probability and $d$ its conditional size-mean difference. Then
\[
 g_G\ge\frac{aq}{1+a-q}
 >\frac{a}{1+a}\bigl[-\log(1-q)\bigr].
\]
Consequently, if a forest has a chosen root in each of its nonempty
components, with root probabilities $q_i$, its deficit is at least
\[
 \frac{a}{1+a}\sum_i-\log(1-q_i)
 =\frac{t}{1+2t}\sum_i-\log(1-q_i).
\]
The empty forest has zero deficit and zero sum. The logarithmic
inequality is strict for every nonempty such forest.
\end{lemma}
\begin{proof}
Root conditioning gives
\[
 g_G=(1-q)g^0+qg^1+q-q(1-q)d^2
 \ge q-q(1-q)d^2,
\]
where $g^0,g^1\ge0$ are the deficits of the two unweighted residual
forests (the occupied root's contribution is the displayed $q$).
Combine this inequality with $g_G\ge aq d^2$. Take weight
$(1-q)/(1+a-q)$ on the lower bound $aqd^2$, and weight
$a/(1+a-q)$ on $q-q(1-q)d^2$. The $d^2$ terms cancel, giving
$g_G\ge aq/(1+a-q)$.

Put $b=1+a$. For $0\le q\le a\le1/2$, define
\[
 f(q)=\frac{bq}{b-q}+\log(1-q).
\]
Then $f(0)=0$, and
\[
 f'(q)=\frac{q\,[b(2-b)-q]}{(b-q)^2(1-q)}>0
 \quad(0<q\le a),
\]
because $b(2-b)-q=1-a^2-q\ge1-a^2-a\ge1/4$.
It follows that
\[
 -\log(1-q)<\frac{bq}{b-q},
\]
which, after multiplication by $a/b$, is the claimed strict
logarithmic floor. The bound $q\le a$ follows from the unweighted
root odds relation. Deficits add over independent components,
proving the forest statement.
\end{proof}

\begin{corollary}[Deletion floor at an unweighted root]\label{cor:four-center-deletion}
For an unweighted rooted tree, let its root odds be $y$, let $d$
be its conditional size-mean difference, and put
\[
 L=\log(t/y)=\sum_i-\log(1-q_i)\ge0,
\]
where $q_i$ are the child-root probabilities after deleting the root.
The deficit $g^0$ of the root-deleted forest satisfies
\[
 g^0\ge\frac{t}{1+2t}L.
\]
Moreover,
\begin{equation}\label{d2four:eq:quarter-deletion-floor}
 g^0\ge y(1-d)^2+\frac{t}{4(1+2t)}L
 \ge y(1-d)^2+
 \frac{t(t-y)}{2(1+2t)(t+y)}.
\end{equation}
A stronger activity-dependent form is
\begin{equation}\label{d2four:eq:activity-deletion-floor}
 g^0\ge y(1-d)^2+
 \frac{t(5-3t)}{8(1+2t)}L
 \ge y(1-d)^2+
 \frac{t(5-3t)(t-y)}{4(1+2t)(t+y)}.
\end{equation}
\end{corollary}
\begin{proof}
The first bound is the preceding lemma applied to the child trees.
When $L>0$, the inherited Cauchy--Schwarz deletion reserve gives
\[
 g^0\ge\frac aL(1-d)^2.
\]
Since $y=te^{-L}$ and $Le^{-L}\le e^{-1}$,
\[
 \frac aL\ge\frac{e}{1+t}y>\frac{8}{3(1+t)}y.
\]
Here $e>8/3$ follows from the positive exponential series beyond its
cubic term. In particular $a/L>(4/3)y$.
Taking three quarters of this last quadratic bound and one quarter
of the logarithmic bound yields the first inequality in
\eqref{d2four:eq:quarter-deletion-floor}.
More generally, take weight $3(1+t)/8$ on the quadratic bound
with coefficient $8y/[3(1+t)]$, and the remaining weight
$(5-3t)/8$ on the logarithmic bound. Both weights belong to $[0,1]$,
and this gives the first inequality in
\eqref{d2four:eq:activity-deletion-floor}.
Finally, for $0<y\le t$, the elementary logarithmic inequality gives
\[
 L=\log(t/y)\ge\frac{2(t-y)}{t+y}.
\]
This follows by differentiation of
$\log u-2(u-1)/(u+1)$ on $u\ge1$.
If $L=0$, the rooted tree is a singleton, $d=1$, $y=t$, and
$g^0=0$. Every displayed bound then holds directly.
\end{proof}

\subsection{Clique assembly and adjacent-pair total variance}
\begin{theorem}[Clique assembly with marked roots]\label{thm:marked-clique}
Let $(T_i,o_i)_{i=1}^k$, with $k\ge1$, be pairwise disjoint nonempty
rooted trees, and assign
fixed multipliers $0<\theta_i\le\Theta$ to their roots. All other vertex
multipliers are one. Add every edge between distinct roots. At common
activity $0<t\le1$, the resulting independent-set size $K$ satisfies
\[
 \operatorname{Var}K<\mathbb EK.
\]
Additional unweighted forest components may be included.
\end{theorem}
\begin{proof}
In each separate marked rooted tree write $q_i,d_i,g_i$ as in
Theorem~\ref{thm:marked-square-reserve}, and let $g_i^0\ge0$ be
the deficit after deleting the root; that graph is an unweighted forest.
Put
\[
 r_i=\frac{q_i}{1-q_i},\qquad Z=1+\sum_i r_i,
 \qquad p_i=\frac{r_i}{Z}\le q_i.
\]
The joining clique contains at most one selected root. Applying total
variance first in each isolated piece and then in the assembly gives
the exact identity
\[
 g=\sum_i\frac{p_i}{q_i}(g_i-q_i^2d_i^2)
   +\left(\sum_i p_id_i\right)^2
   +\sum_i\left(1-\frac{p_i}{q_i}\right)g_i^0.
\]
All terms are nonnegative by the square reserve. If any piece has
more than one vertex, its first summand is strictly positive. Otherwise
all $d_i=1$, and the displayed square is strictly positive.
Additional unweighted forest components have nonnegative deficits,
which add independently.
\end{proof}

\begin{corollary}[Adjacent-pair marking]\label{cor:adjacent-pair-variance}
Let $L$ be a forest with adjacent vertices $u,v$. For $0<t\le1$ and
fixed $0<\theta\le\Theta$, give an independent set $S$ weight
\[
 \theta^{\mathbf1_{u\in S\,\mathrm{or}\,v\in S}}t^{|S|}.
\]
Then $\operatorname{Var}|S|<\mathbb E|S|$.
\end{corollary}
\begin{proof}
Deleting $uv$ splits its tree component into two rooted trees. Mark
both roots by $\theta$ and apply Theorem~\ref{thm:marked-clique}
with two pieces. Since $u,v$ are adjacent, at most one can belong to
$S$, so the product of their two multipliers is exactly the stated
OR-event multiplier. Other components are unweighted and independent.
\end{proof}

\begin{corollary}[Canonical adjacent-pair variance]\label{cor:adjacent-pair-canonical}
For the same adjacent pair, fixed $0<\theta\le\Theta$, and integers
$N\ge\alpha(L)$ and $1\le r\le N$, the law on independent sets
with weights
\[
 \theta^{\mathbf1_{u\in S\,\mathrm{or}\,v\in S}}
 \binom{N-|S|}{r-|S|}
\]
satisfies $\operatorname{Var}|S|<\mathbb E|S|$.
\end{corollary}
\begin{proof}
Apply Lemma~\ref{lem:capacity-transfer} to
Corollary~\ref{cor:adjacent-pair-variance}. Explicitly, at $r=N$ the
binomial multiplier is one. At $r<N$, the multiplier's consecutive
size ratios $(r-k)/(N-k)$ decrease and are at most $r/N$. Match its
positive mean with a unique hard-core activity $t_*\le r/N<1$.
The conditioned law has log-concave likelihood ratio relative to that
law, and the equal-mean two-crossing comparison reduces variance.
The strict hard-core variance inequality proves the conclusion.
\end{proof}

\subsection{A soft edge between two singly marked forest pieces}
Write
\[
 \Theta=\tfrac12\exp(3-2W(1/e)).
\]
We use Theorem~\ref{thm:marked-square-reserve}: for a rooted tree with only
its root marked by a fixed factor in $(0,\Theta]$, at common activity
$0<t\le1$,
\[
 g_i:=\mathbb EK_i-\operatorname{Var}K_i\ge q_i^2d_i^2,
\]
where $q_i$ is its root occupation probability and $d_i$ is its
occupied-minus-unoccupied conditional mean size. This inequality is
strict for a non-singleton tree and equality holds for a singleton.
Deleting the root leaves an unweighted forest, whose deficit $g_i^0$
is nonnegative.

\begin{theorem}[Two-piece soft-edge variance]\label{thm:soft-edge-variance}
Let $(T_1,o_1),(T_2,o_2)$ be disjoint nonempty rooted trees. Give their
roots fixed multipliers $\theta_1,\theta_2\in(0,\Theta]$, and every
other vertex multiplier one. For any fixed $\beta\in[0,1]$, weight
independent sets of the disjoint union by
\[
 \theta_1^{\eta}\theta_2^{\xi}\beta^{\eta\xi}t^{|S|},
 \qquad \eta=\mathbf1_{o_1\in S},\quad\xi=\mathbf1_{o_2\in S},
\]
where $0^0=1$. Then, for $0<t\le1$,
\[
 \operatorname{Var}|S|<\mathbb E|S|.
\]
The same conclusion permits extra unweighted forest components in
both rooted pieces.
\end{theorem}
\begin{proof}
Let $q,s\in(0,1)$ be the root probabilities in the two separate marked
pieces, and put $\delta=1-\beta$ and $Z=1-\delta qs>0$.
Reweighting their independent root indicators by $\beta^{\eta\xi}$
gives marginals and covariance
\[
 p=\frac{q(1-\delta s)}Z\le q,\qquad
 r=\frac{s(1-\delta q)}Z\le s,
\]
\[
 C:=\operatorname{Cov}(\eta,\xi)
 =-\frac{\delta qs(1-q)(1-s)}{Z^2}\le0.
\]
Moreover
\[
 pr-|C|=\frac{\beta qs}Z\ge0.
\]
In particular the quadratic form
\[
 p^2d_1^2+r^2d_2^2-2C d_1d_2
\]
is positive semidefinite: its diagonal entries are nonnegative and
its determinant is $p^2r^2-C^2\ge0$.

The root-state reweighting does not change either piece's conditional
law given its root. Its averaged conditional deficit in the new law
is therefore
\[
 \frac p qg_1+\left(1-\frac p q\right)g_1^0
    +p(1-q)d_1^2
\]
for piece one, and the analogous expression with $p,q,1$ replaced
by $r,s,2$ for piece two. Conditioned on both root states, the pieces
are independent. The variance of their conditional total mean is
\[
 p(1-p)d_1^2+r(1-r)d_2^2+2C d_1d_2.
\]
The law of total variance now gives the exact deficit decomposition
\begin{align*}
 g={}&\frac p q(g_1-q^2d_1^2)
   +\frac r s(g_2-s^2d_2^2)\\
 &+\left(1-\frac p q\right)g_1^0
   +\left(1-\frac r s\right)g_2^0\\
 &+p^2d_1^2+r^2d_2^2-2C d_1d_2.
\end{align*}
Every displayed contribution is nonnegative. If either tree is
non-singleton, its square-reserve slack gives a strictly positive
term, since $p/q,r/s>0$. If both are singletons, $d_1=d_2=1$ and
the final quadratic form equals $p^2+r^2-2C>0$.
Additional unweighted components are independent of this reweighting;
their nonnegative deficits add.
\end{proof}

\subsection{Suppressing the absent--absent boundary event}
\begin{corollary}[Absent--absent event variance]\label{cor:absent-absent}
Let $(H,u),(K,v)$ be disjoint nonempty rooted forests, let
$\tau=\mathbf1_{u\notin S,\ v\notin S}$, and fix
\[
 \Theta^{-1}=2e^{-3+2W(1/e)}\le\gamma\le1.
\]
For $0<t\le1$, the independent-set law with weights
\[
 \gamma^\tau t^{|S|}
\]
has $\operatorname{Var}|S|<\mathbb E|S|$.
\end{corollary}
\begin{proof}
Put $\theta=1/\gamma\le\Theta$. For $\eta,\xi\in\{0,1\}$,
\[
 \gamma^{(1-\eta)(1-\xi)}
   =\gamma\,\theta^{\eta+\xi}\gamma^{\eta\xi}.
\]
The common factor $\gamma$ cancels in normalization. Thus this is
Theorem~\ref{thm:soft-edge-variance} with root marks $\theta$ and
soft-edge factor $\beta=\gamma$.
\end{proof}

\begin{corollary}[Canonical absent--absent variance]\label{cor:absent-canonical}
Under the preceding hypotheses, let $N,r$ be integers with
$N\ge\alpha(H\sqcup K)$ and $1\le r\le N$. The law with weights
\[
 \gamma^\tau\binom{N-|S|}{r-|S|}
\]
satisfies $\operatorname{Var}|S|<\mathbb E|S|$.
\end{corollary}
\begin{proof}
All event weights are positive and fixed. Apply Lemma~\ref{lem:capacity-transfer} to Corollary~\ref{cor:absent-absent}.
For $r=N$ the multiplier is one. Otherwise the size multiplier has
successive ratios $(r-k)/(N-k)$, so its mean matches the hard-core
law at some $0<t_*\le r/N<1$. Its log-concave likelihood ratio and
the equal-mean two-crossing argument reduce variance; the matched
hard-core inequality is strict.
\end{proof}

\begin{corollary}[Low-degree absent--absent interfaces]\label{cor:low-degree-absent}
Let $(H,u),(K,v)$ be disjoint nonempty rooted forests, and suppose
$\deg_H(u),\deg_K(v)\le4$. For every fixed $0<\gamma\le1$ and
$0<t\le1$, the law of independent sets with weights
\[
 \gamma^{\mathbf1_{u\notin S,\,v\notin S}}t^{|S|}
\]
satisfies $\operatorname{Var}|S|<\mathbb E|S|$. The same holds under
canonical completion weights
$\gamma^{\mathbf1_{u\notin S,\,v\notin S}}
 \binom{N-|S|}{r-|S|}$ for $N\ge\alpha(H\sqcup K)$ and $1\le r\le N$.
\end{corollary}
\begin{proof}
Apply the deficit identity in the proof of
Theorem~\ref{thm:soft-edge-variance} with root multipliers
$\theta_1=\theta_2=1/\gamma$ and soft-edge factor $\beta=\gamma$.
Corollary~\ref{cor:low-degree-mark} supplies each marked
piece's square reserve. Root-deleted forests retain nonnegative deficits.
Strictness and independent padding follow exactly as in that identity.
For the canonical conclusion apply Lemma~\ref{lem:capacity-transfer}.
For $r<N$, the matching activity is at most $r/N<1$; for $r=N$,
the binomial multiplier is one and the first conclusion applies at $t=1$.
\end{proof}

% Research frontier continuation, October 26, 2026.
% Depends only on the proved unweighted rooted reserve and marked square reserve.

\section{A uniform-ratio reserve for arbitrary weighted forests}\label{nuf:sec:forest-ratio}
\paragraph{Scope.}
This section proves a dimension-free extension of the common-activity forest
variance theorem. Arbitrary branching and arbitrarily many distinct site
activities are allowed, provided the largest site activity is at most one
and is at most $4/3$ times the smallest. The fixed-mark canonical
consequence follows as well. This does not prove barrier log-concavity
for general forest independence polynomials. The known heterogeneous-star
obstruction has activity ratio $23$ and does not meet this hypothesis.
The argument extends the rooted reserve and logarithmic deletion method
in the preceding probability foundation; it does not assert a
new proof of that earlier equal-activity special case.

\begin{theorem}[A quantitative heterogeneous forest reserve]
\label{nuf:thm:main}
Let $0<t\le1$, $\rho\ge1$, and define
\[
 a=\frac{t}{\rho+t},\qquad
 \delta=2-\frac{\rho+t}{e},\qquad
 B=\frac{t(1+1/e)}{1+t}.
\]
Suppose
\begin{equation}\label{nuf:eq:criterion}
 \delta>0,\qquad
 \frac B\delta+(\rho+2t)e^{-\delta-1}<1.
\end{equation}
Give every vertex $v$ of a finite nonempty forest a site activity
$\lambda_v\in[t/\rho,t]$. Its hard-core independent-set size $K$
satisfies $\Var K<\E K$.

More precisely, for a tree rooted at any vertex, write $q$ for its
root occupation probability and
\[
 d=\E[K\mid\text{root occupied}]
   -\E[K\mid\text{root unoccupied}].
\]
Then
\begin{equation}\label{nuf:eq:reserve}
 \E K-\Var K\ge aqd^2.
\end{equation}
The inequality is strict for every tree with at least two vertices.
For a singleton, equality holds precisely when its activity is $t/\rho$.
Condition \eqref{nuf:eq:criterion} holds for every $0<t\le1$ when
$1\le\rho\le4/3$.
\end{theorem}

\begin{lemma}[Scalar inequality with separate upper and lower activities]
\label{nuf:lem:scalar}
Under \eqref{nuf:eq:criterion}, put $A_0=t/(1+t)$ and let
$0<q_i\le A_0$, $1\le i\le k$, where $k\ge1$. Set
\[
 L_i=\log\frac{t(1-q_i)}{q_i},\qquad
 L=\sum_i-\log(1-q_i),\qquad S=\sum_iq_i.
\]
Let $\lambda\in[t/\rho,t]$, define $q/(1-q)=\lambda e^{-L}$,
and write $s=1-q$. For arbitrary real $d_i$, subject to $d_i=1$
whenever $L_i=0$, put $u=\sum_iq_id_i$. Then
\[
 as\sum_iq_id_i^2+aq\sum_{i:L_i>0}\frac{(1-d_i)^2}{L_i}
 +q-q(s+a)(1-u)^2>0.
\]
\end{lemma}
\begin{proof}
Use the notation
\[
 M=\sum_i\frac{q_iq}{q+s q_iL_i},\qquad
 V=\frac{S-M}{as},\qquad H=asM,\qquad C=s+a,\qquad b=qC.
\]
Completing squares in the $d_i$ and applying weighted
Cauchy--Schwarz gives, when $V>0$,
\[
 as\sum_iq_id_i^2+aq\sum_i\frac{(1-d_i)^2}{L_i}
 \ge H+\frac{(u-M)^2}{V}.
\]
An index with $L_i=0$ contributes $q_i$ to $M$, zero to $V$,
and $asq_i$ to $H$; it has the required fixed value $d_i=1$.
If $V=0$, all the $L_i$ vanish, and $u=M=S$ is fixed.
It therefore suffices to prove
\[
 F=(q+H)(1-bV)-b(1-M)^2>0.
\]
Indeed this also implies $1-bV>0$, and for $V>0$ the remaining
one-variable quadratic has minimum $F/(1-bV)$. For $V=0$ its
constrained value equals $F$.

Put $A=as$, $w=t/e$, and $R=qC/A$. Direct expansion yields
\[
 F=q(q-a)-\frac{q^2CS}{A}
       +M\left[A+qC\left(2+\frac qA-S\right)\right].
\]
The message bounds are
\begin{equation}\label{nuf:eq:bounds}
 cL\le S\le L,\quad c=\frac{A_0}{\log(1+t)},\qquad
 q_iL_i\le w,\qquad
 M\ge\frac{Sq}{q+sw}.
\end{equation}
Here the first bound uses the increasing function
$-\log(1-x)/x$ on $(0,1)$, and the second uses
$x\log(t/x)\le t/e$.
The separate upper and lower activity hypotheses give
\begin{equation}\label{nuf:eq:new-compatibility}
 R\le(\rho+2t)e^{-L},\qquad
 a-q\le A(1-e^{-L}),\qquad
 q+sw\le A_0+(1-A_0)w=B.
\end{equation}
For the first inequality, $q/s=\lambda e^{-L}$ and
$t(1+a)/a=\rho+2t$. For the second,
\[
 \frac{a-q}{as}
 =1-\frac{\rho\lambda}{t}e^{-L}\le1-e^{-L}.
\]
For the last, $q\le A_0$ and $w<1$.

The coefficient $P$ of $M$ is positive. If $S\le2$, this is immediate.
Otherwise
\[
 \frac PA\ge1-R(S-2)
 \ge1-(\rho+2t)(S-2)e^{-S}
 \ge1-(\rho+2t)e^{-3}>0.
\]
The last strict inequality follows from \eqref{nuf:eq:criterion}, since
$0<\delta<2$ implies
$(\rho+2t)e^{-3}<(\rho+2t)e^{-\delta-1}<1$.
Substitute the lower bound for $M$ in \eqref{nuf:eq:bounds}. Exact
simplification gives
\[
 F\ge\frac q{q+sw}G,\qquad
 G=(q-a)(q+sw)+S\left[A+qC\left(2-S-\frac wa\right)\right].
\]
Since $w/a=(\rho+t)/e$, \eqref{nuf:eq:new-compatibility} implies
\begin{equation}\label{nuf:eq:G}
 \frac GA\ge S-(1-e^{-L})(q+sw)+RS(\delta-S).
\end{equation}
If $S\le\delta$, the last term is nonnegative and
\[
 \frac GA\ge S\left(1-\frac Bc\right)
 =S\left(1-(1+1/e)\log(1+t)\right)>0.
\]
For example, the elementary estimates $e>163/60$ and
$\log2<7/10$ bound the factor subtracted from one above by
$1561/1630<1$.
If $S>\delta$, then \eqref{nuf:eq:G} gives
\[
 \frac G{AS}
 >1-\frac B\delta-(\rho+2t)e^{-\delta-1}>0,
\]
using $L\ge S$ and
$\max_{S\ge\delta}(S-\delta)e^{-S}=e^{-\delta-1}$.
Thus $G>0$ and $F>0$ in both cases.
\end{proof}

\begin{proof}[Proof of Theorem~\ref{nuf:thm:main}]
Fix $t,\rho$, and the site activities throughout. Induct on the
number of vertices, asserting the reserve at every root. For a singleton,
$d=1$ and $g:=\E K-\Var K=q^2\ge aq$, since $q\ge a$.
Now let $T_i$ be the child trees of a non-singleton rooted tree,
with rooted parameters $q_i,d_i$, deficits $g_i$, and root-deleted
deficits $g_i^0$. Independence after conditioning on the top root gives
\[
 \frac q{1-q}=\lambda\prod_i(1-q_i),\qquad
 d=1-\sum_iq_id_i,
\]
\[
 g=s\sum_i g_i+q\sum_i g_i^0+q-qs d^2.
\]
The induction hypothesis gives $g_i\ge aq_id_i^2$.
Let $q_{ij},d_{ij}$ be the corresponding parameters of the children
of the root of $T_i$, whose site activity is $\lambda_i$. Then
\[
 L_i^{\rm true}:=\log\frac{\lambda_i(1-q_i)}{q_i}
   =\sum_j-\log(1-q_{ij})\le L_i,
 \qquad
 1-d_i=\sum_jq_{ij}d_{ij}.
\]
The root-deleted graph is a forest of strictly smaller trees, so
Cauchy--Schwarz and induction yield
\[
 g_i^0\ge a\sum_jq_{ij}d_{ij}^2
 \ge a\frac{(1-d_i)^2}{L_i}
\]
when $L_i>0$. If $T_i$ is a singleton the numerator vanishes and
the inequality is immediate, even when $L_i>0$; otherwise the
display follows first with the positive denominator $L_i^{\rm true}$.
If $L_i=0$, then $q_i=A_0$ forces $\lambda_i=t$ and no child vertices,
so $d_i=1$ exactly. Subtracting $aqd^2$ from the mixture formula and
applying Lemma~\ref{nuf:lem:scalar} proves strictness for every
non-singleton. Every nonempty tree has positive deficit, and deficits
add over forest components.

It remains to verify the stated explicit ratio range. The left side of
\eqref{nuf:eq:criterion} is increasing in $t$ and $\rho$ whenever
$\delta>0$, so it suffices to take $t=1$, $\rho=4/3$. The elementary
series bound
$e>1+1+1/2+1/6+1/24+1/120=163/60$ gives
\[
 \delta>\frac{186}{163},\qquad B<\frac{223}{326}.
\]
Using $e^x>1+x+x^2/2$ at $x=23/163$ gives
\[
 \frac B\delta+\frac{10}{3}e^{-\delta-1}
 <\frac{223}{372}
 +\frac{10}{3}\left(\frac{60}{163}\right)^2
  \left(1+\frac{23}{163}+\frac12\left(\frac{23}{163}\right)^2\right)^{-1}
 =\frac{4513559}{4550676}<1.
\]
The logarithm estimate used in the lemma follows from
$e^{7/10}>1+7/10+(7/10)^2/2+(7/10)^3/6>2$.
\end{proof}

\begin{corollary}[Fixed marks and canonical completion]
\label{nuf:cor:canonical}
Let $G$ be a nonempty finite forest and give its vertices fixed marks
$\theta_v>0$ with
\[
 b=\max_v\theta_v\le1,\qquad
 \frac{\max_v\theta_v}{\min_v\theta_v}\le\frac43.
\]
For every $0<x\le1$, the law with weights
$x^{|S|}\prod_{v\in S}\theta_v$ is strictly sub-Poisson.
For all integers $N\ge\alpha(G)$ and $1\le r\le N$, the same is
true of the law on independent sets with weights
\[
 \left(\prod_{v\in S}\theta_v\right)
 \binom{N-|S|}{r-|S|}.
\]
More generally the canonical conclusion holds for $b>1$ whenever
$br/N\le1$.
\end{corollary}
\begin{proof}
At activity $x$, the site activities lie in $[3bx/4,bx]$.
Theorem~\ref{nuf:thm:main} applies whenever $bx\le1$. This same ratio
bound persists while $x$ decreases: none of the fixed marks is
changed in the mean-matching argument.

For completeness, the consecutive ratios of the completion multiplier
are $(r-k)/(N-k)\le r/N$ when $r<N$. Hence the completed size law
has mean no greater than the fixed-mark law at activity $r/N$.
The latter mean is continuous, strictly increasing with activity,
and tends to zero at zero. There is therefore a unique
$0<x_*\le r/N$ matching the positive completed mean. Relative to
the law at $x_*$, the completion likelihood ratio is log-concave
with interval support. Equal mass and equal mean then make the
completed variance no larger: the difference of the two probability
mass functions is positive on an interval and nonpositive outside,
and summing against $(k-u)(k-v)$ at the two crossing points proves
the second-moment comparison. Since $bx_*\le br/N\le1$, the
reference variance is strictly smaller than its mean. If $r=N$,
the completed law is exactly the fixed-mark law at activity one;
the stated condition requires $b\le1$ and the theorem applies directly.
\end{proof}

\begin{corollary}[Weighted forest coronas are strictly log-concave]
\label{nuf:cor:corona}
Let $H$ be a nonempty finite forest, give its vertices fixed marks
$0<\theta_v\le1$ whose largest-to-smallest ratio is at most $4/3$,
and let $N\ge |V(H)|$ be an integer. Define
\[
 Q(x)=(1+x)^N\sum_{S\text{ independent in }H}
    \left(\prod_{v\in S}\theta_v\right)
    \left(\frac{x}{1+x}\right)^{|S|}
   =\sum_{j=0}^N B_jx^j.
\]
Then for every internal $1\le j<N$,
\begin{equation}\label{nuf:eq:corona-curvature}
 \frac{B_j^2}{B_{j-1}B_{j+1}}>
 \frac{N-j+1}{N-j}>1.
\end{equation}
In graph terms, arbitrary fixed base marks in this range preserve
strict log-concavity for a forest leaf corona, with unweighted private
leaves and arbitrary isolated-vertex padding. The same conclusion holds
for empty $H$, when $Q(x)=(1+x)^N$.
\end{corollary}
\begin{proof}
Write $c_k=\sum_{|S|=k}\prod_{v\in S}\theta_v$ and $n=|V(H)|$.
If $H$ is empty, $Q(x)=(1+x)^N$ and the result follows directly.
Otherwise fix $r=j+1$ and let $\nu$ be the size law proportional to
$c_k\binom{N-k}{r-k}$. Write $\mu=\E_\nu K$, $v=\Var_\nu K$,
and $m=r-\mu$. Corollary~\ref{nuf:cor:canonical} gives $v<\mu$.
Weighted insertion counting gives
\[
 (k+1)c_{k+1}
 =\sum_{\substack{S\text{ independent}\\|S|=k}}
   \left(\prod_{u\in S}\theta_u\right)
   \sum_{w:\,S\cup\{w\}\text{ independent},\,w\notin S}\theta_w
 \le(n-k)c_k.
\]
Because $N\ge n$, multiplication by the consecutive completion
ratio $(r-k)/(N-k)$, for $k<N$, yields
$(k+1)\nu(k+1)\le(r-k)\nu(k)$. At $k=N$, both sides vanish
(with zero extension), so no division is needed. Summing gives $\mu\le r/2$.
The two consecutive binomial identities give exactly
\[
 \frac{B_j}{B_{j+1}}=\frac m{N-j},\qquad
 \frac{B_{j-1}}{B_{j+1}}
 =\frac{m^2-m+v}{(N-j)(N-j+1)}.
\]
All $B_j$ are positive because of the empty-set term $(1+x)^N$,
so the denominator $m^2-m+v$ is positive. Consequently
\[
 \frac{B_j^2}{B_{j-1}B_{j+1}}
 =\frac{N-j+1}{N-j}\frac{m^2}{m^2-m+v}
 >\frac{N-j+1}{N-j},
\]
since $m-v>r-2\mu\ge0$.
\end{proof}

\paragraph{The larger sufficient ratio furnished by this proof.}
For all $0<t\le1$, one may replace $4/3$ by any fixed $\rho\ge1$
satisfying
\[
 \frac{1+1/e}{2\{2-(\rho+1)/e\}}
 +(\rho+2)e^{-3+(\rho+1)/e}<1.
\]
The expression is strictly increasing until its denominator vanishes.
Its unique value-one crossing is approximately $1.3538$ (the exact
root definition is authoritative). The endpoint itself is not claimed.
This is a sufficient constant from the estimates, not the optimal
heterogeneous-tree ratio. At smaller upper activity the two-parameter
criterion \eqref{nuf:eq:criterion} is stronger. Neither this theorem nor
its canonical consequence asserts arbitrary-tree homogeneous quadratic
compensation, arbitrary two-cut BL, or unrestricted forest BL.

\section{A doubled rooted reserve for every nontrivial tree}
\label{cld:sec:doubled-reserve}
The results in this section strengthen the ordinary rooted reserve of
Lemma~\ref{lem:forest-subpoisson} for every unweighted tree with at least
two vertices. The proof retains a local conditional-deficit envelope and uses three
attachment lemmas followed by strong induction. The logarithmic and rational
deletion floors used below are Lemma~\ref{lem:four-root-floor} and
Corollary~\ref{cor:four-center-deletion}; those proofs use the earlier
ordinary reserve, so there is no circular dependence on the doubled reserve.

\subsection{Optimal local retention of the two conditional deficits}
\label{cld:sec:optimal-local-retention}
The following two-variable optimization retains the relation between
the two conditional deficits. Its use in the preservation lemmas is
pointwise at the fixed activity.

Let a rooted piece have root occupation probability $q\in(0,1)$,
conditional size-mean difference $d$, and deficit
$g=\mathbb EK-\operatorname{Var}K$. Write $g^0$ for the deficit after
root deletion and $g^1$ for the deficit after deleting the closed root
neighborhood. Thus the occupied-root conditional deficit is $1+g^1$.
The exact mixture identity is
\begin{equation}\label{cld:eq:optimal-local-mixture}
 g=(1-q)g^0+q(1+g^1)-q(1-q)d^2.
\end{equation}
The root-deleted and closed-neighborhood-deleted graphs are unweighted
forests when only the root carries a fixed mark. In particular $g^1\ge0$.

\begin{proposition}[Optimal local conditional-deficit envelope]
\label{cld:prop:optimal-local-envelope}
Suppose
\[
 g\ge\kappa qd^2,\qquad g^0\ge\ell(d),\qquad g^1\ge0,
\]
where $\ell(d)\ge0$. For any new root marginal $0\le p\le q$, define
\[
 \mathcal B_{p,q,\kappa,\ell}(d)=
 \max\left\{(1-p)\ell(d)+p,
 \left(1-\frac pq\right)\ell(d)
       +p(\kappa+1-q)d^2\right\}.
\]
Then
\[
 (1-p)g^0+p(1+g^1)\ge\mathcal B_{p,q,\kappa,\ell}(d).
\]
Moreover, this is the best lower bound that follows from exactly the
three displayed inequalities and the mixture identity, with $d,p,q$
and $\kappa$ fixed and the two conditional deficits otherwise free.
This optimality is for the stated relaxation; it does not claim that
all minimizing data are realized by rooted trees.
\end{proposition}
\begin{proof}
Put $x=g^0$, $y=1+g^1$, and
$A=q(\kappa+1-q)d^2$. The constraints become
\[
 x\ge\ell(d),\qquad y\ge1,\qquad (1-q)x+qy\ge A.
\]
For a fixed $x$, the minimizing choice is
$y=\max\{1,A/q-(1-q)x/q\}$. Hence the objective equals
\[
 \max\left\{(1-p)x+p,
 \left(1-\frac pq\right)x+\frac pq A\right\}.
\]
Both affine functions are nondecreasing in $x$, because $p\le q$.
The minimum is therefore attained at $x=\ell(d)$ and is precisely
$\mathcal B_{p,q,\kappa,\ell}(d)$.
\end{proof}

For disjoint rooted pieces coupled only through their root states, let
$p_i$ be the new marginals and $C$ their root-state covariance. Conditional
independence gives the exact identity
\[
 g_{\rm assembled}=\sum_i[(1-p_i)g_i^0+p_i(1+g_i^1)]-d^\top C d.
\]
Whenever $p_i\le q_i$, the proposition gives a valid lower bound for
each conditional contribution. Convex combinations of its two branches
remain valid; no two lower bounds are simply added.

\subsection{A proved leaf-extension preservation step}
\begin{theorem}[Same-root leaf extension of the doubled reserve]\label{cld:thm:same-root-leaf-doubled-reserve}
Fix $0<t\le1$ and write $a=t/(1+t)$. Let $(T,o)$ be a nonempty
unweighted rooted tree, with root occupation probability $q$, conditional
size-mean difference $d$, and deficit $g=\mathbb EK-\operatorname{Var}K$.
Assume the stronger rooted reserve
\[
 g\ge2a qd^2.
\]
Attach one new private leaf to $o$ and retain $o$ as the root. The new
tree has root probability $p$ and conditional response $d'=d-a$, and
its deficit satisfies
\[
 g_{\mathrm{new}}>2a p(d')^2.
\]
This is a preservation statement conditional on the displayed reserve
for $T$. It does not assert that every nontrivial rooted tree already
satisfies that reserve, nor does it cover re-rooting at the new leaf.
\end{theorem}
\begin{proof}
Let $g^0$ and $g^1$ be the deficits after deleting the old root and
its closed neighborhood, respectively. The latter is nonnegative.
The established root-deletion estimate gives
\[
 g^0\ge r(1-d)^2,\qquad r=\frac q{1-q}.
\]
For a singleton this is interpreted on its actual response $d=1$;
otherwise it follows from the unweighted Cauchy--Schwarz deletion
reserve and $Le^{-L}\le e^{-1}$.

The old root and the new leaf cannot both be occupied. Their new
marginals are
\[
 p=\frac{q(1-a)}{1-aq},\qquad
 v=\frac{a(1-q)}{1-aq}.
\]
Apply Proposition~\ref{cld:prop:optimal-local-envelope} with
$\kappa=2a$ and choose outer convex weight
\[
 \eta=\frac{2a}{1+2a-q}.
\]
Since $0<q\le a\le1/2$, this belongs to $(0,1)$.
The identity $(1-\eta)(1+2a-q)=1-q$ shows that the old piece's
conditional contribution, after its marginal changes to $p$, is at least
\[
 p(1-q)d^2+(\alpha+\eta p)(1-d)^2+\eta p,
 \qquad
 \alpha=r-(1+r)p=tp,
 \quad t=\frac a{1-a}.
\]
The leaf's contribution is exactly $v$. The root-state covariance
has diagonal $p(1-p),v(1-v)$ and off-diagonal $-pv$. Therefore
\begin{align*}
 g_{\mathrm{new}}
 &\ge p(1-q)d^2+p(t+\eta)(1-d)^2+\eta p
       -p(1-p)d^2+v^2+2pvd.
\end{align*}
Root conditioning gives $d'=d-a$ exactly. Subtracting its target
reserve and collecting terms yields
\[
 g_{\mathrm{new}}-2ap(d-a)^2
 \ge p\,[Ad^2+2Bd+C],
\]
where
\[
 A=t+\eta+p-q-2a,\qquad
 B=-t-\eta+v+2a^2,\qquad
 C=t+2\eta+\frac{v^2}{p}-2a^3.
\]
We now certify $A>0$ and $AC-B^2>0$ throughout
$0<q\le a\le1/2$.

For the interior substitute
\[
 a=\frac{x}{2(1+x)},\qquad q=\frac{ay}{1+y},
 \qquad x,y>0.
\]
Define the positive polynomials
\[
 R=3xy+4x+2y+2,\qquad
 S=3x^2y+4x^2+8xy+8x+4y+4.
\]
Exact expansion gives
\[
 A=\frac{P_A(x,y)}{2(x+1)(x+2)(y+1)RS},
\]
\[
 AC-B^2=\frac{P_D(x,y)}
 {8y(x+1)^4(x+2)^2(y+1)R^2S}.
\]
The two numerator certificates are
\[
 \begin{array}{c|c|r|r}
 &\text{bidegree}&\text{positive integer coefficients}&\text{minimum}\\\hline
 P_A&(5,3)&20&9\\
 P_D&(10,5)&54&111
 \end{array}
\]
Every nonzero coefficient is positive. The complete arrays are in
\path{leaf_extension_certificate.json}; the corresponding script
reconstructs the rational expressions and verifies the inverse
substitution by polynomial cross-multiplication. Thus both quantities
are strictly positive in the interior.

The activity-one face $a=1/2$, parameterized by
$q=y/[2(1+y)]$, has
\[
 A=\frac{3y^2+6y+4}{2(y+1)(3y+4)},\qquad
 AC-B^2=
 \frac{37y^4+176y^3+316y^2+240y+64}
 {8y(y+1)(3y+4)^2}>0.
\]
For completeness the face $q=a$, parameterized by
$a=x/[2(1+x)]$, has
\[
 A=\frac{x(3x^2+4x+4)}{2(x+1)(x+2)(3x+2)},
\]
\[
 AC-B^2=
 \frac{x^2(37x^5+130x^4+196x^3+184x^2+112x+32)}
 {8(x+1)^4(x+2)(3x+2)^2}>0.
\]
At their intersection the values are $A=1/2$ and $AC-B^2=37/72$.
All required boundary cases are therefore included.

It follows that the quadratic $Ad^2+2Bd+C$ is strictly positive
for every real $d$. Since $p>0$, the desired strict reserve follows.
\end{proof}

\subsection{A new-leaf root preserves the doubled reserve}
\begin{theorem}[New-leaf-root preservation]
\label{cld:thm:new-leaf-doubled-reserve}
Fix $0<t\le1$ and put $a=t/(1+t)$. Let $(Y,o)$ be a nontrivial
unweighted rooted tree with root probability $q$, conditional size-mean
difference $d$, and deficit $g_Y=\mathbb EK_Y-\operatorname{Var}K_Y$.
Assume $g_Y\ge2a qd^2$. Attach a new leaf $z$ to $o$ and choose $z$
as root. Its root probability $p$ and response $d_z=1-qd$ satisfy
\[
 g_{Y+z}>2a p(1-qd)^2.
\]
For singleton $Y$, the conclusion holds directly for the resulting
$K_2$, without assuming a doubled reserve on the singleton.
\end{theorem}
\begin{proof}
Write the old root odds as $s=tv^2$, where $0<v<1$ because $Y$
is nontrivial, and put
\[
 Z=1+t+s,\quad q=s/(1+s),\quad p=t/Z,\quad h=s/Z,
 \quad b=t/(1+2t),\quad L=-\log(v^2).
\]
Thus $p,h$ are the new marginals of the leaf and old root.
The elementary logarithmic inequalities imply
\[
 L\ge4(1-v)/(1+v),\qquad L\le(1-v^2)/v.
\]
For the latter, differentiate $v^{-1}-v+2\log v$, whose derivative
is $-(1-v)^2/v^2$ and whose value at one is zero.
The former is the standard inequality
$\log u\ge2(u-1)/(u+1)$ applied to $u=1/v$.
Combining half of each established deletion reserve
$g_Y^0\ge a(1-d)^2/L$ and $g_Y^0\ge bL$ gives
\[
 g_Y^0\ge J(1-d)^2+F,
 \qquad J=\frac{av}{2(1-v^2)},\qquad
 F=\frac{2b(1-v)}{1+v}.
\]
Apply Proposition~\ref{cld:prop:optimal-local-envelope} with $\kappa=2a$ and
outer weight $\eta=2a/(2a+1-q)$. Its deletion coefficient is
\[
 W=\eta(1-h)+(1-\eta)(1-h/q)=\frac{t+\eta}{Z}.
\]
Since $(1-\eta)(2a+1-q)=1-q$, the old piece contributes at least
\[
 W[J(1-d)^2+F]+\eta h+h(1-q)d^2
\]
to the conditional deficit. The new leaf contributes exactly $p$.
The two root indicators have covariance $-ph$ and variances
$p(1-p),h(1-h)$. Root conditioning gives $d_z=1-qd$ exactly.

Put $B=WJ$, $C_*=WF+\eta h$, and $T=p(h+2aq)$.
Using $h-q=-pq$ and subtracting the target reserve yields
\[
 g_{Y+z}-2ap(1-qd)^2\ge(d,1)\mathcal N(d,1)^\top,
 \qquad
 \mathcal N=\begin{pmatrix}
 B-qT&T-B\\T-B&B+C_*+p^2-2ap
 \end{pmatrix}.
\]
Its leading minors have exact factorizations
\[
 \mathcal N_{11}=t^2vP_1/D_1,
 \qquad\det\mathcal N=t^4vP_2/D_2,
\]
where
\begin{align*}
 D_1={}&2(t+1)(1-v)(1+v)(tv^2+1)^2
 (tv^2+t+1)^2(2t^2v^2+3t+1),\\
 D_2={}&2(t+1)(2t+1)(1-v)(1+v)^2(tv^2+1)^2
 (tv^2+t+1)^3(2t^2v^2+3t+1)^2.
\end{align*}
Both denominators are positive. The complete tensor Bernstein
certificates on $[0,1]^2$, in variable order $(t,v)$, are
\[
 \begin{array}{c|c|r|c}
 &\text{bidegree}&\text{positive coefficients}&\text{minimum}\\\hline
 P_1&(5,9)&60&3\\
 P_2&(7,14)&120&3503/429
 \end{array}
\]
Every coefficient is strictly positive, with no zero or negative
entries. Thus both polynomials are positive on the entire closed
square, including $t=1$. The exact rational arrays are
\path{new_leaf_bernstein_minor_1.json} and
\path{new_leaf_bernstein_minor_2.json}. The reproducible
\path{new_leaf_certificate.py} verifies the conditional quadratic,
the rational minors, Bernstein conversion and its inverse, and
agreement with a separate construction. The independent probability
and certificate audit is \path{audit_newleaf.py}.
Sylvester's criterion gives $\mathcal N\succ0$, proving strictness.

For singleton $Y$, the new graph is $K_2$, with
$p=t/(1+2t)$, response $1-a$, and deficit $4p^2$. Directly,
\[
 4p^2-2ap(1-a)^2
 =\frac{2t^2(1+4t+6t^2+2t^3)}{(1+2t)^2(1+t)^3}>0.
\]
This treats the singleton without an invalid doubled-reserve premise.
\end{proof}

\subsection{Binary gluing preserves the doubled reserve at the old root}

\begin{theorem}[Binary gluing with a nontrivial attached piece]
\label{cld:thm:binary-gluing-doubled-reserve}
Fix $0<t\le1$ and put $a=t/(1+t)$. Let $(X,o_X)$ and $(Y,o_Y)$
be disjoint nontrivial unweighted rooted trees. Write $q$ for the
root probability of $X$, $d_X$ for its conditional size-mean
difference, and $g_X=\mathbb EK_X-\operatorname{Var}K_X$.
Assume
\[
 g_X\ge2a qd_X^2.
\]
Join $o_X$ to $o_Y$ by one edge, retaining $o_X$ as the root, and
write $p,d',g_{X\ast Y}$ for the new root probability, conditional
response, and deficit. Then
\[
 g_{X\ast Y}>2a p(d')^2.
\]
Only the ordinary previously established reserve is required for
$Y$. This is a preservation theorem conditional on the displayed
hypothesis for $X$; an arbitrary-tree base for that hypothesis is
not asserted in this theorem.
\end{theorem}

\begin{proof}
Let $r,s$ be the separate root odds and put
\[
 q=\frac r{1+r},\qquad w=\frac s{1+s},\qquad
 Z=1+r+s,\qquad p=\frac rZ,\qquad h=\frac sZ.
\]
Nontriviality gives $0<r,s<t$, so write
\[
 r=tu^2,\qquad s=tv^2,\qquad 0<u,v<1.
\]
The joined root-state law has weights $1,r,s$ on $00,10,01$.
Its mean is $(p,h)$ and its covariance is
\[
 C=\begin{pmatrix}p(1-p)&-ph\\-ph&h(1-h)\end{pmatrix}.
\]
The separate response of $Y$ is denoted by $d_Y$.
Root conditioning gives the exact response identity
\begin{equation}\label{cld:eq:binary-joined-response}
 d'=d_X-wd_Y.
\end{equation}
Indeed, conditioning the new old root to be absent leaves the whole
piece $Y$ available, whereas conditioning it to be occupied forces
$o_Y$ absent. The resulting change in the mean size of $Y$ is
$-wd_Y$.

For either separate piece, let $g^0,g^1$ be the deficits after
root deletion and closed-neighborhood deletion, respectively.
Its occupied-root conditional deficit is $1+g^1$. Write
\[
 B_X=(1-p)g_X^0+p(1+g_X^1),\qquad
 B_Y=(1-h)g_Y^0+h(1+g_Y^1).
\]
Conditional independence given the joined root states gives
\begin{equation}\label{cld:eq:binary-exact-mixture}
 g_{X\ast Y}=B_X+B_Y-(d_X,d_Y)C(d_X,d_Y)^\top.
\end{equation}
All uses of Proposition~\ref{cld:prop:optimal-local-envelope} below are
pointwise inequalities at the fixed common activity $t$.

Set
\[
 A_X=\frac{au}{1-u^2},\qquad
 A_Y=\frac12\left(\frac{av}{1-v^2}+s\right),\qquad
 F_Y=\frac{t(5-3t)(1-v^2)}{8(1+2t)(1+v^2)}.
\]
The established Cauchy--Schwarz deletion reserve is
$g^0\ge a(1-d)^2/L$, with $L=\log(t/r)$ or $\log(t/s)$.
For $0<z<1$,
\[
 -2\log z\le z^{-1}-z,
\]
so $1/(-2\log z)\ge z/(1-z^2)$. Consequently
\[
 g_X^0\ge A_X(1-d_X)^2.
\]
For $Y$, average this rational Cauchy--Schwarz bound with the
previously established full rational deletion bound
\[
 g_Y^0\ge s(1-d_Y)^2+
 \frac{t(5-3t)(t-s)}{4(1+2t)(t+s)}.
\]
The result is
\begin{equation}\label{cld:eq:binary-y-deletion}
 g_Y^0\ge A_Y(1-d_Y)^2+F_Y.
\end{equation}
These deletion estimates require only the earlier ordinary
unweighted rooted reserve, not the doubled reserve being preserved.

For $X$, retain the second branch of the conditional-deficit
envelope with $\kappa=2a$. Since $1-p/q=h$, this gives
\[
 B_X\ge hA_X(1-d_X)^2+p(2a+1-q)d_X^2.
\]
For $Y$, the ordinary reserve is $g_Y\ge awd_Y^2$, and
$g_Y^1\ge0$ because the closed-neighborhood-deleted graph is an
unweighted forest. Apply its envelope with $\kappa=a$, and take
weight
\[
 \eta=\frac a{a+1-w}\in(0,1)
\]
on the occupied-root-baseline branch and weight $1-\eta$ on the
other branch. Since $(1-\eta)(a+1-w)=1-w$, the resulting lower
bound, using \eqref{cld:eq:binary-y-deletion}, is
\[
 B_Y\ge W\bigl[A_Y(1-d_Y)^2+F_Y\bigr]
       +h(1-w)d_Y^2+\eta h,
\]
where
\[
 W=1-\frac hw+\eta\frac{h(1-w)}w
   =\frac{r+\eta}{Z}>0.
\]
Substituting these bounds in \eqref{cld:eq:binary-exact-mixture} and
subtracting the target reserve from \eqref{cld:eq:binary-joined-response}
yields
\[
 g_{X\ast Y}-2ap(d_X-wd_Y)^2
 \ge(d_X,d_Y,1)G(d_X,d_Y,1)^\top,
\]
with the symmetric matrix
\begin{equation}\label{cld:eq:binary-gluing-matrix}
G=\begin{pmatrix}
 h(A_X-pq)&p(h+2aw)&-hA_X\\
 p(h+2aw)&WA_Y-phw-2apw^2&-WA_Y\\
 -hA_X&-WA_Y&hA_X+W(A_Y+F_Y)+\eta h
\end{pmatrix}.
\end{equation}
For example, the first diagonal follows from
$p-q=-hq$ and the second from $h-w=-pw$.
Thus no covariance or target-response term has been omitted.

We certify $G\succ0$ throughout $0<t\le1$, $0<u,v<1$.
Let $\Delta_j$ be its leading principal determinant of order $j$.
Put
\[
 Z=1+tu^2+tv^2,\qquad H=1+2t+t^2v^2.
\]
Exact polynomial identities give
\[
 \Delta_1=\frac{t^2uv^2P_1}
 {(1+t)(1+tu^2)(1-u^2)Z^2},
\]
\[
 \Delta_2=\frac{t^4uv^3P_2}
 {2(1+t)^2(1+tu^2)(1+tv^2)^2(1-u^2)(1-v^2)Z^3H},
\]
\[
 \Delta_3=\frac{t^6uv^3P_3}
 {16(1+t)^3(1+tu^2)(1+tv^2)^2(1-u^2)(1-v^2)
       (1+2t)(1+v^2)Z^4H^2}.
\]
Every displayed denominator and monomial factor is strictly
positive on the stated domain. The reduced integer polynomials
$P_1,P_2,P_3$ have respectively
\[
 \begin{array}{c|c|r}
 &\text{multidegree in }(t,u,v)&\text{nonzero power coefficients}\\\hline
 P_1&(2,5,2)&9\\
 P_2&(7,9,11)&208\\
 P_3&(11,11,17)&943
 \end{array}
\]
and are supplied in
\path{binary_sparse_minor_1.json} through
\path{binary_sparse_minor_3.json}.

In the tensor Bernstein basis of these multidegrees, $P_1$ has
54 strictly positive coefficients on the entire closed unit cube,
with minimum coefficient one. For $P_2$ and $P_3$, restrict to the
following dyadic boxes; in every row $t$ ranges over $[0,1]$.
Each local restriction is affinely pulled back to the unit cube.
The exact coefficient minima are
\[
\begin{array}{c|c|c|r|c}
 &u\text{-interval}&v\text{-interval}&\text{positive coefficients}
 &\text{minimum coefficient}\\\hline
 P_2&[0,1/2]&[0,1]&960&1\\
 P_2&[1/2,1]&[0,1/2]&960&416893/9461760\\
 P_2&[1/2,1]&[1/2,1]&960&1554899/1376256\\\hline
 P_3&[0,1/2]&[0,1]&2592&4\\
 P_3&[1/2,3/4]&[0,1/2]&2592&43787104003/9358540800\\
 P_3&[3/4,1]&[0,1/2]&2592&3125/256\\
 P_3&[1/2,1]&[1/2,1]&2592&1725/128
\end{array}
\]
All coefficients in each local tensor array are strictly positive;
there are no omitted zero or negative coefficients. The listed boxes
cover the closed unit cube, with disjoint interiors. Since the tensor
Bernstein basis is nonnegative and sums to one on a closed cube,
these coefficient minima prove strict positivity everywhere,
including $t=1$ and every common subdivision face. It follows that
all three $\Delta_j$ are strictly positive on the required domain,
and Sylvester's criterion gives $G\succ0$.

For reproducibility, \path{binary_sparse_certificate.py}
constructs the reduced rational determinant identities;
\path{convert_binary_bernstein.py} constructs and exactly
inverts the global basis conversion; and
\path{subdivide_binary_certificate.py} constructs the dyadic
restrictions. Their exported arrays and reports record the complete
integer data. The independent script
\path{audit_binary_certificate.py} starts from the conditional
contributions in \eqref{cld:eq:binary-exact-mixture}, clears matrix rows
independently, and reconstructs all three determinants. It then
inverts the Bernstein arrays by finite differences, repeats the
subdivision by midpoint de Casteljau recursion, and verifies every
restricted polynomial against a direct affine substitution of the
original power polynomial. It also checks full cube coverage and
all strict coefficient minima. All these checks use exact rational
or integer arithmetic.

Finally $(d_X,d_Y,1)$ is nonzero, so its quadratic form under $G$
is strictly positive. This proves the claimed strict reserve.
\end{proof}

\Needspace{12\baselineskip}
\subsection{Strong induction and its two-root consequence}
\begin{theorem}[Doubled reserve on every nontrivial rooted tree]
\label{cld:thm:doubled-root-reserve}
Let $T$ be a finite unweighted tree with at least two vertices, choose
any root $o$, and fix $0<t\le1$. Put
\[
 a=\frac{t}{1+t},\quad q_o=\Pr(o\in S),\quad
 d_o=\mathbb E[|S|\mid o\in S]-\mathbb E[|S|\mid o\notin S].
\]
Then
\[
 \mathbb E|S|-\operatorname{Var}|S|>2a q_od_o^2.
\]
\end{theorem}
\begin{proof}
Use strong induction on $|V(T)|$, with the root arbitrary. The
$K_2$ base is the direct calculation in
Theorem~\ref{cld:thm:new-leaf-doubled-reserve}.

Delete the chosen root $o$ and let $Y_1,\ldots,Y_k$ be its nonempty
child trees, rooted at their neighbors of $o$; $k\ge1$. Each has
strictly fewer vertices than $T$. Start with the tree $X$ consisting
of $o$, $Y_1$, and their connecting edge. If $Y_1$ is a singleton,
$X=K_2$ supplies the doubled reserve at $o$. If $Y_1$ is nontrivial,
the induction hypothesis supplies its doubled reserve at the child
root, and Theorem~\ref{cld:thm:new-leaf-doubled-reserve} supplies the
strict doubled reserve for $X$ rooted at $o$.

Adjoin the remaining child trees one at a time, always retaining
$o$ as root. For a singleton child, use Theorem~\ref{cld:thm:same-root-leaf-doubled-reserve}. For a nontrivial child, use Theorem~\ref{cld:thm:binary-gluing-doubled-reserve}: the current root piece has the
doubled reserve, and the appended unweighted child has the ordinary
rooted reserve. Every operation preserves the strict doubled reserve
at $o$. The resulting tree is $T$, completing the induction.
\end{proof}

\begin{corollary}[Two-root additive reserve in one tree]
\label{cld:cor:two-root-additive-reserve}
Let $T$ be a finite unweighted tree with at least two vertices, and
let $u,v$ be two distinct vertices. At common activity $0<t\le1$,
write $a=t/(1+t)$ and let $q_w,d_w$ denote the root probability and
conditional size-mean response when $w$ is chosen as root. Then
\[
 \mathbb E|S|-\operatorname{Var}|S|
 >a\bigl(q_ud_u^2+q_vd_v^2\bigr).
\]
\end{corollary}
\begin{proof}
The deficit $g=\mathbb E|S|-\operatorname{Var}|S|$ does not depend
on the choice of root. The doubled-reserve theorem gives
$g>2a q_ud_u^2$ and $g>2a q_vd_v^2$. Adding these inequalities and
dividing by two proves the claim.
\end{proof}

This corollary concerns two roots of a single ordinary unweighted
tree. It does not itself establish endpoint-interaction variance,
unrestricted two-cut BL, or general forest BL.

\paragraph{Exact proof data.}
The source subtree \path{anc/closure_2026_10_26/frontier_closure/}
contains the three certificate generators and complete rational arrays.
The same-root leaf certificate is under \path{policy/}; the general
binary-gluing certificate is under \path{conditional/}. Independent
reconstruction and probability audits are under \path{audit_binary/}
and \path{audit_newleaf/}. Every certificate verifies an identity on
its complete parameter domain. Finite tree scans are supporting
diagnostics, not the basis of the strong-induction theorem.

\section{Limits of heterogeneous variance preservation}\label{sec:heterogeneous-limit}

\begin{proposition}[A weighted star is already super-Poisson]\label{prop:weighted-star-obstruction}
Let $T=K_{1,67}$. Give its center activity $1$ and every leaf activity
$1/23$. If $K$ is the hard-core independent-set size, then
\[
 \mathbb EK-\operatorname{Var}K
 =\frac{67U^2-1206U+576}{576(1+U)^2}<0,
 \qquad U=(24/23)^{67}.
\]
Thus a hard-core law on a tree with all activities at most $1$ need
not be sub-Poisson. In particular, the ordinary weighted-path
compensation theorem cannot be extended to arbitrary tree skeletons
under only the same bounds on activities or directed messages.
\end{proposition}
\begin{proof}
More generally take a star with $k$ leaves, center activity $1$, and
common leaf activity $r>0$. Put
\[
 U=(1+r)^k,\qquad a=\frac r{1+r},\qquad
 q=\frac1{1+U},\qquad s=1-q.
\]
The center is occupied with probability $q$. Given that it is
occupied the size is $1$, whereas given that it is absent the size
has the binomial law $\operatorname{Bin}(k,a)$. The conditional
mean and variance formulas therefore give
\begin{align*}
 \mathbb EK&=q+ska,\\
 \operatorname{Var}K&=ska(1-a)+qs(1-ka)^2,\\
 \mathbb EK-\operatorname{Var}K
 &=\frac{1+Uk\{a^2(U+1-k)+2a\}}{(1+U)^2}.
\end{align*}
For $k=67,r=1/23$, this is the displayed expression.
The integer inequality
\[
 2\cdot24^{67}<35\cdot23^{67}
\]
proves $1<U<35/2$. The convex polynomial
$f(u)=67u^2-1206u+576$ satisfies
$f(1)=-563$ and $f(35/2)=-41/4$; it is therefore strictly negative
throughout this interval. The supplied verifier checks the integer
inequality and the moments using exact rational arithmetic.

For completeness, let $B$ be the occupancy vector, $p=\mathbb EB$,
$C=\operatorname{Cov}B$, and set
$\alpha_i=r_i(1-p_i)-p_i$. The proposed tree compensation form is
\[
 Q(d,h)=d^\top(\operatorname{diag}p-C)d+
              \sum_i\alpha_i(h-d_i)^2.
\]
At $d=\mathbf1,h=1$ it equals
$\mathbb EK-\operatorname{Var}K<0$. Thus the failure is an actual
realizable conditional-law obstruction, not merely a negative
direction in an enlarged parameter relaxation. All directed
cavity occupation probabilities are at most $1/2$: a leaf-to-center
message is $1/24$, and a center-to-leaf message is
$[1+(24/23)^{66}]^{-1}<1/2$.
\end{proof}

\begin{proposition}[An infinite family with nonvanishing deficit]\label{prop:weighted-star-family}
For every integer $m\ge25$, give the center of $K_{1,3m}$ activity
$1$ and its leaves activity $1/m$. Its size law satisfies
$\operatorname{Var}K>\mathbb EK$. Moreover
\[
 \lim_{m\to\infty}(\mathbb EK-\operatorname{Var}K)
 =\frac{1-3e^3}{(1+e^3)^2}<0.
\]
\end{proposition}
\begin{proof}
With $U=(1+1/m)^{3m}$, the preceding formula becomes
\[
 g_m=\frac{1+\dfrac{3mU}{(m+1)^2}(U+3-m)}{(1+U)^2}.
\]
For $m\ge2$, the first three binomial terms give
$(1+1/m)^m\ge2+(m-1)/(2m)\ge9/4$, so $U>10$.
Also $U<e^3<21$. Hence for $m\ge25$ the numerator of $g_m$ is
strictly less than
\[
 1-\frac{30m(m-24)}{(m+1)^2}<0.
\]
The last inequality is equivalent to
$29m^2-722m-1>0$; at $m=25$ it equals $74$, and it increases
thereafter. The limit follows from $U\to e^3$ in the exact formula.
\end{proof}

\paragraph{Scope.}
The activities are unequal. These examples do not contradict the
proved sub-Poisson theorem for an unweighted forest at one common
activity $0<t\le1$. Nor do they falsify ordinary-path compensation,
forest barrier log-concavity, or the log-concavity of an unweighted
independence polynomial. They show that an arbitrary-tree assembly
argument must retain additional compatibility or reserve information;
site activity and message bounds by themselves are insufficient.

\section{Further questions}
The compatibility condition in Theorem~\ref{nuf:thm:main} is sufficient,
and its optimal activity ratio remains undetermined. The weighted-star
examples show that a universal heterogeneous theorem needs additional
control, but they do not locate the sharp boundary of the compatible
region. The common-activity theorem applies throughout $(0,1]$;
the present proof does not claim the same rooted reserve at arbitrary
activity.

For coefficients, the factorial normalization in
Theorem~\ref{thm:allorder-wc} is uniform at every index, whereas the
binomial ULC prefix is exactly three indices. Determining optimal
finite-size curvature margins remains a separate question. The
weighted-corona theorem permits arbitrary branching under compatible
base marks; unrestricted site marks are excluded by the variance
counterexamples above.

All statements in this article have complete analytic or algebraic
proofs. The accompanying checks reproduce identities and boundary
constants; finite tests are not substitutes for the quantified proofs.

\section*{Use of AI tools}
ChatGPT was used for mathematical exploration and proof development,
code preparation and checking, literature comparison, and language
editing. The author is responsible for the mathematical content,
attribution, and computational materials. This statement does not
assert formal verification.

\end{document}